\documentclass{birkmult}

\usepackage{graphicx}
\usepackage{tikz}
\usepackage{amssymb}
\usepackage{mathrsfs}
\usepackage{amsmath}
\usepackage{amsfonts}
\usepackage{amsthm}
\usepackage[english]{babel}
\usepackage{mdwlist}
\usepackage{comment}

\usepgflibrary{decorations.markings}
\usetikzlibrary{decorations.markings}
\usepackage{hyperref}

\DeclareMathOperator{\supp}{supp}
\DeclareMathOperator{\im}{Im}
\DeclareMathOperator{\re}{Re}

\DeclareMathOperator{\diag}{diag}
\DeclareMathOperator{\transpose}{T}

\newcommand{\ds}{\displaystyle}

\newtheorem{theorem}{Theorem}[section]
\newtheorem{lemma}[theorem]{Lemma}
\newtheorem{proposition}[theorem]{Proposition}

 \theoremstyle{definition}
 
 \theoremstyle{remark}

 \numberwithin{equation}{section}

\theoremstyle{definition}
\newtheorem{definition}[theorem]{Definition}

\theoremstyle{remark}
\newtheorem*{example}{Example}
\newtheorem{remark}[theorem]{Remark}

\begin{document}
%
%
%
%
%
%
%
%
%
\
\title{Double scaling limit for  modified Jacobi-Angelesco polynomials}

\date{Dedicated to the memory of Julius Borcea}
\author{Klaas Deschout}

\address{Department of Mathematics \br
Katholieke Universiteit Leuven \br
Celestijnenlaan 200 B\br
3001 Leuven, Belgium}

\email{klaas.deschout@wis.kuleuven.be}

\author{Arno B.J. Kuijlaars}
\address{Department of Mathematics \br
Katholieke Universiteit Leuven \br
Celestijnenlaan 200 B\br
3001 Leuven, Belgium}
\email{arno.kuijlaars@wis.kuleuven.be}


\date{September 29, 2010}
\dedicatory{Dedicated to the memory of Julius Borcea}

\maketitle

\begin{abstract}
We consider multiple orthogonal polynomials with respect to  two modified
Jacobi weights on touching intervals $[a,0]$ and $[0,1]$, with $a < 0$,
and study a transition that occurs at $a = -1$.
The transition is studied in a double scaling limit,
where we let the degree $n$ of the polynomial tend to infinity  while the
parameter $a$ tends to $-1$ at a rate of $O(n^{-1/2})$.
We obtain a Mehler-Heine type asymptotic formula for the polynomials
in this regime.
The method used to analyze the problem is the steepest descent
technique for Riemann-Hilbert problems.  A key point in the
analysis is the construction of a new local parametrix.
\end{abstract}

\section{Introduction and statement of results}

\subsection{Introduction}

Multiple orthogonal polynomials are a generalization of orthogonal polynomials
that originated in works on Hermite-Pad\'e rational approximation problems, but recently
found other applications in random matrix theory and related probabilistic models.

In the approximation theory literature two main classes of multiple orthogonal
polynomials were identified for which detailed asymptotic results are available.
These are the Angelesco systems and the Nikishin systems. In an Angelesco system \cite{ang19} the multiple
orthogonality is defined on disjoint intervals, while in a Nikishin system \cite{niso91}
the orthogonality is on the same interval with orthogonality measures that are related to each other
via an intricate hierarchical structure.

A main stimulus for the asymptotic analysis of orthogonal polynomials was given
by the formulation of a $2 \times 2$ matrix valued Riemann-Hilbert problem for orthogonal
polynomials by Fokas, Its and Kitaev \cite{fik92} and the subsequent  application
of the powerful Deift-Zhou steepest descent technique to this Riemann-Hilbert problem
in \cite{dkmvz99a,dkmvz99b} and many later papers.

A Riemann-Hilbert problem for multiple orthogonal polynomials was formulated
by Van Assche, Geronimo and Kuijlaars \cite{vgk01}. The Riemann-Hilbert problem is
of size $(r+1) \times (r+1)$, where $r$ is the number of orthogonality
weights for the multiple orthogonal polynomials. The Riemann-Hilbert formulation
was already used in several papers,
see e.g.~\cite{abk05, akv08, aply10, blku04, blku07, bff09, duku09, kmw09, kvw05, lywi08}
for the asymptotic analysis of multiple orthogonal polynomials and their
associated multiple orthogonal polynomial ensembles \cite{kui10a, kui10b}.

In this paper we consider Angelesco systems on two touching intervals
$[a,0]$ and $[0,1]$ with $a < -1$. Our interest is in the special behavior
at $0$ that takes place near a critical value of $a$. A prime example for
this situation is given by the Jacobi-Angelesco weights
\begin{equation}
\begin{aligned} \label{JacobiAngelesco}
    w_1(x) & = |x-a|^{\alpha} |x|^{\beta} |x-1|^{\gamma}, \quad && x \in (a,0), \\
  w_2(x) & = |x-a|^{\alpha} |x|^{\beta} |x-1|^{\gamma}, && x \in (0,1),
  \end{aligned}
  \end{equation}
 with $\alpha, \beta, \gamma > -1$, which were first studied by Kaliaguine \cite{kal79, karo96}.
The associated multiple orthogonal polynomials are among the classical
multiple orthogonal polynomials \cite{abv03} and as such have
a number of very special properties. There is e.g.\ a  raising operator which gives rise to a Rodrigues-type
formula and a third order linear differential equation as well as an explicit
four term recurrence relation for the diagonal case
Jacobi-Angelesco multiple orthogonal polynomials, see \cite{apt89,  karo96, tak05, tak09, tul09, vas99, vaco01}.

\subsection{Modified Jacobi-Angelesco weights}

We generalize the system \eqref{JacobiAngelesco} by considering more general
modified Jacobi weights on the two intervals $(a,0)$ and $(0,1)$. We will use the following
weights $w_1$ and $w_2$ throughout this paper.

\begin{definition}
Let $a < 0$, $\alpha, \beta, \gamma > -1$ and define
\[ \Delta_1 = [a,0], \qquad \Delta_2 = [0,1]. \]
For $j=1,2$, let $h_j$ be strictly positive on $\Delta_j$ with an analytic
continuation to a neighborhood of $\Delta_j$ in the complex plane.
Then we define
\begin{equation}
\begin{aligned}
    w_{1}(x) & = (x-a)^{\alpha} |x|^{\beta} h_{1}(x), \quad && x \in \Delta_1, \\
    w_{2}(x) & =  x^{\beta} (1-x)^{\gamma} h_{2}(x), && x \in \Delta_2.
    \label{definitionweights}
\end{aligned}
\end{equation}
\end{definition}
When appropriate we set $w_j(x) \equiv 0$ for $x \in \mathbb R \setminus \Delta_j$.

The definition of the multiple orthogonal polynomial  (of type II) with respect to the
weights \eqref{definitionweights} is as follows.
\begin{definition}
Given a multi-index $(n_{1}, n_{2}) \in \mathbb{N}^{2}$ the multiple orthogonal polynomial
is defined as the unique monic polynomial $P_{n_{1},n_{2}}$ of degree $n_{1} + n_{2}$ such that
\begin{equation}
    \int_{\Delta_j} P_{n_{1},n_{2}}(x) x^k w_j(x) \, dx = 0, \quad
    \textrm{for } k = 0, \mathellipsis n_{j} -1,
    \label{definitionpn1n2} \end{equation}
for $j=1,2$.
    \end{definition}
Since we are dealing with an Angelesco system of weights \cite{ang19} the polynomial
$P_{n_{1},n_{2}}$ indeed exists and is uniquely characterized by \eqref{definitionpn1n2}. It is also
known that all the zeros of $P_{n_{1},n_{2}}$ are real and simple with  $n_1$ zeros in $(a,0)$ and $n_2$
zeros in $(0,1)$, see e.g.\ \cite{vas06}.
For the definition of the multiple orthogonal polynomials of type I we also refer to \cite{vas06}.

\subsection{The phase transition}
\label{subsectionphasetransition}

We consider in this paper the diagonal case
\[ n_1 = n_2 = n. \]
It is known that the zeros of the multiple orthogonal polynomial $P_{n,n}$
have a weak limit as $n \to \infty$, which only depends on the parameter $a < 0$.
The limiting zero distribution can be characterized as the solution to a vector
equilibrium problem for two measures \cite{gora81,niso91}.

Define the logarithmic energy $I(\nu)$ of a measure $\nu$ as
\begin{equation} \label{logenergy}
    I(\nu) := \iint  \log \frac{1}{|x-y|} \, d\nu(x) d\nu(y),
\end{equation}
and the mutual logarithmic energy $I(\nu, \mu)$ of two measures $\nu$ and $\mu$ as
\begin{equation} \label{mutuallogenergy}
    I(\nu, \mu) := \iint \log \frac{1}{|x-y|} \, d\nu(x) d\mu(y).
\end{equation}
Then the vector equilibrium problem is defined as follows.

\begin{definition} \label{vectorequilibrium}
The vector equilibrium problem asks to minimize the energy functional
\begin{equation} \label{energy}
    E(\nu_{1},\nu_{2}) := I(\nu_{1}) + I(\nu_{1},\nu_{2}) + I(\nu_{2})
\end{equation}
among positive measures $\nu_{1}$ and $\nu_{2}$ with $\supp(\nu_1) \subset [a,0]$, $\supp(\nu_2) \subset [0,1]$
and $\int d\nu_{1} = \frac{1}{2}$, $\int  d\nu_{2} = \frac{1}{2}$
\end{definition}

One may interpret this energy functional as the energy resulting
from two conductors $[a,0]$ and $[0,1]$ with each an equal amount of charged particles.
Particles on the same conductor repel each other, such that the resulting electrostatic
force is proportional to the inverse of the distance between the two particles,
which accounts for the terms $I(\nu_{1})$ and $I(\nu_{2})$.
Additionally, particles on different conductors also repel each other,
but with only half the strength. This leads to the term $I(\nu_{1},\nu_{2})$ in \eqref{energy}.
This kind of interaction is known as Angelesco-type interaction, see \cite{apt98}.

The minimizers $\nu_{1}$ and $\nu_{2}$ for the Angelesco equilibrium problem
are called the equilibrium measures.  They exist, are unique and are absolutely continuous with respect to the Lebesgue measure.
It is due to Kaliaguine \cite{kal79} that the endpoints of the supports of the equilibrium measures are
given by $a,0,1$ and a fourth point $b$:
\begin{equation}
    b = \frac{(a+1)^{3}}{9(a^{2}-a+1)} \label{expressions}
    \end{equation}
such that
\begin{equation}
    \begin{aligned}
     \supp \nu_{1} & = [a,b] \subset [a,0], \, && \supp \nu_{2} = [0,1], \quad &&  \textrm{ if } a \leq -1, \\
     \supp \nu_{1} & = [a,0], && \supp \nu_{2} = [b,1] \subset [0,1], && \textrm{ if } a \geq -1.
     \end{aligned}
\end{equation}
We see here the pushing effect: the charge on  the smaller interval pushes away the charge
on the larger interval, thereby creating a gap in the support. The gap disappears in
the symmetric case $a=-1$ where we have $b=0$.

The density of the equilibrium measures blows up as an inverse square root at the endpoints
$a,0$, and $1$ of its supports. These are the so-called hard edges.
For $a \neq -1$ there is a soft edge at $b$, where the equilibrium density vanishes like a square root.
In the symmetric case $a = -1$, where both intervals have equal size, both measures have full supports
and the densities behave like inverse cube roots at $0$. We may call $0$ in this case  a Kaliaguine point
after \cite{kal79}. A sketch of the densities in the three cases is given in Figure~\ref{sketchtransition}.

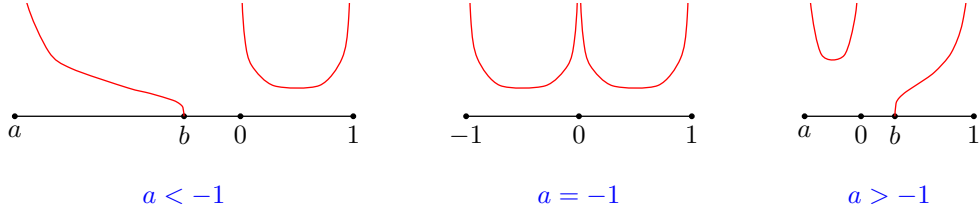
\begin{figure}[t]
\centering
\begin{tikzpicture}[scale = 1.5,line width = .5]
\useasboundingbox (-5,-1) rectangle (4.5,1);
\draw (2.75,-0.7) node[text = blue] {$a > -1$};
\draw (0,-0.7) node[text = blue] {$a = -1$};
\draw (-3.5,-0.7) node[text = blue] {$a < -1$};
\clip (-7,-1) rectangle (5.5,1);
\draw [color = black] (-1,0) -- (1,0);
\filldraw (-1,0) circle (.02) node[below] {$-1$};
\filldraw (0,0) circle (.02) node[below] {$0$};
\filldraw (1,0) circle (.02)  node[below] {$1$};
\draw[color = red, line width = .5] plot[smooth] coordinates{
(-.99, 1.933780056)
(-.96, 0.886107499)
(-.90, 0.5128480254)
(-.70, 0.2731187430)
(-.30, 0.2697788644)
(-.10, 0.4495865000)
(-.04, 0.6779079288)
(-.01, 1.192546280)
(-.001, 2.8)};
\draw[color = red, line width = .5] plot[smooth] coordinates{
(.99, 1.933780056)
(.96, 0.886107499)
(.90, 0.5128480254)
(.70, 0.2731187430)
(.30, 0.2697788644)
(.10, 0.4495865000)
(.04, 0.6779079288)
(.01, 1.192546280)
(.001, 2.8)};
\draw[black] (-5,0) -- (-2,0);
\filldraw (-5,0) circle (.02) node[below] {$a$};
\filldraw (-3,0) circle (.02) node[below] {$0$};
\filldraw (-2,0) circle (.02)  node[below] {$1$};
\filldraw (-3.5,0) circle (.02) node[below]{$b$};
\draw[color = red, line width = .5] plot[smooth] coordinates{
(-2.01, 1.933780056)
(-2.04, 0.886107499)
(-2.10, 0.5128480254)
(-2.30, 0.2731187430)
(-2.70, 0.2697788644)
(-2.90, 0.4495865000)
(-2.96, 0.6779079288)
(-2.99, 1.192546280)
(-2.999, 2.8)
};
\draw[color = red, line width = .5] plot[smooth] coordinates{
(-4.99, 3.314498463)
(-4.96,1.648843638)
(-4.90,1.031937189)
(-4.70, 0.5732656751)
(-4.50,0.4242640687)
(-4,0.2522689246)
(-3.80, 0.2026800233)
(-3.60, 0.1425795491)
(-3.54,0.1110349815)
(-3.51,0.07771916399)
(-3.5,0)
};
\draw[black] (2,0) -- (3.5,0);
\filldraw (2,0) circle (.02) node[below] {$a$};
\filldraw (2.5,0) circle (.02) node[below] {$0$};
\filldraw (3.5,0) circle (.02)  node[below] {$1$};
\filldraw (2.8,0) circle (.02) node[below] {$b$};
\draw[color = red, line width = .5] plot[smooth] coordinates{
(2.005, 3.867560112)
(2.02,1.772214998)
(2.05,1.025696051)
(2.15,0.5462374860)
(2.35,0.5395577288)
(2.45, 0.8991730000)
(2.48,1.355815858)
(2.495,2.385092560)};
\draw[color = red, line width = .5] plot[smooth] coordinates{
(2.8,0)
(2.81,0.1142080481)
(2.84,0.1651445647)
(2.9,0.2177938588)
(3,0.2837224827)
(3.15,0.3900355956)
(3.3,0.5640904640)
(3.4,0.8349473050)
(3.46,1.352001716)
(3.49,2.734220451)
};
\end{tikzpicture}
\caption{A sketch of the equilibrium densities in the three cases $a < -1$, $a=-1$, and $-1 < a < 0$.
For a clearer picture the size of the gap between $0$ and $b$ has been exaggerated.  In a true plot $b$
would be much closer to $0$ as the gap is less than one ninth of the length of the larger interval.}
\label{sketchtransition}
\end{figure}

It is this interior soft-to-hard edge transition as $a$ varies around $-1$
that will give rise to a new critical behavior of the multiple orthogonal polynomials
around $0$ that we wish to describe in this paper.

\subsection{Main result}

The main result of this paper is a Mehler-Heine type asymptotic formula for
the multiple orthogonal polynomial $P_{n,n}(z)$ near $z=0$, with the parameter $a$ near $-1$.
We use $P_{n,n}(z;a)$ to denote the dependence on $a$.

\begin{theorem} \label{maintheorem}
For $a < 0$ close enough to $-1$, let $P_{n,n}(z;a)$ be the multiple
orthogonal polynomial with respect to the weights \eqref{definitionweights} and the multi-index $(n,n)$.  Let $\tau \in \mathbb R$ and
\begin{equation} \label{an}
    a_n = -1 + \frac{\sqrt{2} \tau}{n^{1/2}}.
\end{equation}
Then we have for every $z \in \mathbb C$,
\begin{equation} \label{MehlerHeine}
    P_{n,n} \left( \frac{z}{\sqrt{2} n^{3/2}} ; a_n \right)
    =  (-1)^n C_n Q(z; \tau) \left( 1 + \mathcal O(n^{-1/6})\right)
\end{equation}
where $C_n$ is a positive constant and
\begin{equation} \label{contourQ}
    Q(z; \tau) = i \int_{\Gamma_0}
      t^{-\beta-1} \exp\left(- \frac{z^2}{2t^2} - \frac{\tau z}{t} + t\right) \, dt,
      \qquad z \in \mathbb C,
      \end{equation}
where the contour $\Gamma_0$ is shown in of Figure~{\rm\ref{figuregamma0}}.
The convergence in \eqref{MehlerHeine} is uniform for $z$ in compact subsets of $\mathbb C$.
\end{theorem}

\begin{figure}[t]
\begin{center}
\begin{tikzpicture}[scale = 1.5, line width = 1]
\filldraw (0,0) circle (.03) node[above left]{$0$};
\draw [dashed] (0,0) -- (-1.5,0);
\draw [postaction = decorate, decoration = {markings, mark = at position .5 with {\arrow{stealth};}}] (.5,.8) arc (90:-90:.5 and .8);
\draw (.5,.8) .. controls (0,.8) and (-.5,.5) .. (-1,.5);
\draw (.5,-.8) .. controls (0,-.8) and (-.5,-.5) .. (-1,-.5);
\draw [dotted] (-1,.5) -- (-1.3,.5);
\draw [dotted] (-1,-.5) -- (-1.3,-.5);
\end{tikzpicture}
\caption{The contour $\Gamma_{0}$ appearing in the Mehler-Heine formula \eqref{MehlerHeine}.
The dashed line denotes the cut of $t^{-\beta -1}$.\label{figuregamma0} }
\end{center}
\end{figure}
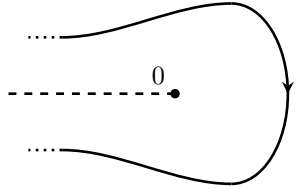

The function $Q$ in \eqref{contourQ} is an entire solution of the third order differential equation
\begin{equation}
    z^{2} Q'''(z) + 2(\beta +1) z Q''(z) + (\beta^{2} + \beta - \tau z) Q'(z) + (z-\tau \beta) Q(z) = 0 \label{diffeqQ}
\end{equation}
The differential equation has a regular singular point at $z=0$, with associated
Frobenius indices equal to $0$, $-\beta$ and $-\beta +1$.
For $\beta > -1$, there is a one-dimensional space of entire solutions to \eqref{diffeqQ}, unless $\beta = 0$ in which case this space is two-dimensional.  In case $\beta \neq 0$ we may characterize $Q$ as the unique entire solution of $\eqref{diffeqQ}$ satisfying
\begin{equation}
Q(0) = \frac{2 \pi}{\Gamma(\beta +1)}.
\end{equation}



The constant $C_{n}$ in \eqref{MehlerHeine} is given by
\begin{equation} \label{constantCn}
C_{n} = \frac{ e^{c_{1} + c_{2}}} {\sqrt{3\pi}} \frac{2^{\beta}}{3^{\frac{ \alpha + 2 \beta + \gamma }{2}}}  e^{-\tau^{2}} n^{\beta + \frac{1}{2}} e^{-\sqrt{2} \tau n^{\frac{1}{2}}} \left( \frac{4}{27}\right)^{n}
\end{equation}
where $c_{1}$ and $c_{2}$ are positive constants defined in \eqref{formulacj} below.
They are determined by the analytic factors $h_1$ and $h_2$ in the weights \eqref{definitionweights},
and are independent of $n$.  For simple analytic factors $h_{1}$ and $h_2$ one can evaluate $c_{1}$ and $c_{2}$ explicitly.
For example, if $h_{k}$ is a constant function, then $c_{k} = 0$.

\begin{remark}
In the case $\tau = 0$  the function $Q(z;\tau)$ from \eqref{contourQ} can be written
as a generalized hypergeometric function
\[ Q(z;0) = \frac{2\pi}{\Gamma(\beta+1)} \mathop{{}_{0}} \! F_2 \! \left( - ; \frac{\beta+1}{2}, \frac{\beta+2}{2} ; -\frac{z^2}{8} \right). \] 
This function was already found by  Sorokin \cite{sor89} in a Mehler-Heine formula
for certain  multiple orthogonal polynomials of Laguerre-type. More recently, it was
obtained for Jacobi-Angelesco multiple orthogonal polynomials by  Tulyakov \cite{tul09} and
Takata \cite{tak09}, who both prove Theorem \ref{maintheorem} 
for the case $\tau = 0$ and weights \eqref{definitionweights} with $h_1 \equiv 1$, $h_2 \equiv 1$. 
\end{remark}

\begin{remark}
In \cite{vaco01} an explicit formula for the Jacobi-Angelesco polynomial $P_{n,n}^{(\alpha,\beta,\gamma)}$ is given, that is,
the multiple orthogonal polynomial with weights \eqref{JacobiAngelesco}, namely
\begin{multline}
\binom{3n + \alpha + \beta + \gamma}{n} P_{n,n}^{(\alpha, \beta, \gamma)}(z;a)  \\
  = \sum_{k=0}^{n} \sum_{j=0}^{n-k} \binom{n + \alpha }{k} \binom{n + \beta}{j} \binom{n + \gamma}{n-k-j} (z-a)^{n-k} z^{n-j} (z-1)^{k+j}.
\end{multline}
Applying Stirling's approximation formula to the binomial coefficients, one can then derive
that with $a_n$ given by  \eqref{an},
\begin{multline}
    P_{n,n}(0;a_{n}) = (-1)^{n} \frac{2\pi}{\Gamma(1+\beta)}  \frac{1}{\sqrt{3\pi}} \left( \frac{2}{3}\right)^{\alpha + \beta + \gamma} e^{-\tau^{2}} \\
        \times n^{\beta + \frac{1}{2}} e^{-\sqrt{2}\tau n^{\frac{1}{2}}} \left(\frac{4}{27}\right)^{n}    \left(1 + \mathcal{O}\left(n^{-\frac{1}{2}}\right)\right)
\end{multline}
which is consistent with \eqref{constantCn}, since in this case one can evaluate $c_{1}$ and $c_{2}$ to be
\begin{equation}
c_{1} =  \gamma (\log 2 - \tfrac{1}{2} \log 3), \qquad c_{2} = \alpha( \log 2 - \tfrac{1}{2} \log 3).
\end{equation}
\end{remark}

\subsection{Overview of the rest of the paper}

We use two main tools to prove the Theorem~\ref{maintheorem}, namely Riemann-Hilbert (RH) problems
and modified equilibrium problems. These will be discussed in the next two sections.

The RH problems are of size $3 \times 3$. We first discuss the RH problem for the multiple
orthogonal polynomial $P_{n_1,n_2}$ with the modified Jacobi weights. The steepest descent
analysis of the paper will lead to a local parametrix that is built out of a local model RH problem,
that is discussed in detail in Section \ref{sec:localRHP}.
This model RH problem is new,
although it is related to another model RH problem studied recently in a different
connection \cite{kmw10}.

The modified equilibrium problem is related to a Riemann surface in Section \ref{sec:modeq}.
The same Riemann surface will also play a role in the construction of the outer parametrix
in the steepest descent analysis.

Section \ref{sec:sda} is the bulk of the paper. It contains the steepest descent
analysis of the RH problem for multiple orthogonal polynomials. It follows the usual
steps in such an analysis as e.g.\ done in \cite{dei99, dkmvz99a,kmvv04}. In a first
transformation we use the $g$-functions coming from the modified equilibrium problem to
normalize the RH problem at infinity. The next transformation is the opening
of lenses. Then we construct outer and local parametrices that are used in
the next transformation. After this transformation one typically arrives at a RH problem
that is normalized at infinity, and for which the jump matrices all tend to the identity matrix
as $n \to \infty$. It is a curious fact that this does not happen in the present
paper. The jump matrix on a circle around $0$ will take the form
\[ I + Z_n(z) + \mathcal{O}(n^{-1/6}) \]
where $Z_n(z)$ is bounded on the circle, but it does not tend to $0$ as $n \to \infty$.
We can resolve this problem by making another transformation, where we use the special
structure of the matrices $Z_n(z)$. This extra step in the steepest descent analysis
is also used in the recent papers \cite{dkz10,kmw10}, which makes it reasonable to suspect
that the need for such an extra step is a more common phenomenon in the
steepest descent analysis of larger size RH problems
in a critical situation.

The proof of Theorem \ref{maintheorem} is given in the final section \ref{sec:proof}.
Here we unravel all the previous transformations, and we pay special attention
to the behavior around $0$.

In a forthcoming paper we plan to analyze the determinantal point process
that is associated with the modified Jacobi-Angelesco weights. This is an example
of a multiple orthogonal polynomial ensemble \cite{kui10a} where
half of the particles are on $[a,0]$ and the other half are on $[0,1]$. There is
again a critical behavior at $0$ as $a$ varies around $-1$, and we will find
a new family of limiting correlation kernels in this setting that are also
related to the solution of the local model RH problem.

\section{First tool: RH problems}

\subsection{The Riemann-Hilbert problem}

The multiple orthogonal polynomial \eqref{definitionpn1n2} are characterized in terms of
a $3 \times 3$ matrix valued Riemann-Hilbert problem (RH problem) due to \cite{vgk01}. We use the RH problem for the
asymptotic analysis to derive our results.

We work with the modified Jacobi weights $w_{1}$ and $w_{2}$ \eqref{definitionweights}
and we take a general multi-index $(n_{1}, n_{2})$.
The RH problem then asks for a function $Y : \mathbb C \setminus [a,1] \to \mathbb C^{3 \times 3}$ such that
\begin{itemize}
\item $Y$ is analytic on $\mathbb{C} \setminus [a,1]$,
\item $Y$ has continuous boundary values $Y_{\pm}$ on $(a,0)$ and $(0,1)$ satisfying a jump relation $Y_{+} = Y_{-} J_{Y}$ with jump matrix
\begin{equation} \label{jY}
    J_{Y}(x) = \begin{pmatrix} 1 & w_{1}(x) & w_{2}(x) \\ 0 & 1 & 0 \\ 0 & 0 & 1 \end{pmatrix},
        \quad x \in (a,0) \cup (0,1), \end{equation}
       where it is understood that $w_1(x) \equiv 0$ on $(0,1)$ and $w_2(x) \equiv 0$ on $(-a,0)$,
\item $Y$ has the asymptotic behavior
\begin{equation}
        Y(z) = \left( I + \mathcal{O}\left(\frac{1}{z} \right)\right)
            \begin{pmatrix} z^{n_1+n_2} & 0 & 0 \\ 0 & z^{-n_1} & 0 \\ 0 & 0 & z^{-n_2} \end{pmatrix}
\end{equation}
as $z \to \infty$, and
\item $Y$ has the following behavior at the endpoints of the intervals
\begin{multline} \label{Yneara}
    Y(z) = \mathcal{O} \begin{pmatrix} 1 & \epsilon(z) & 1 \\
        1 & \epsilon(z) & 1 \\ 1 & \epsilon(z) & 1 \\
        \end{pmatrix}, \textrm{ as } z  \to a,  \\
        \textrm{where } \epsilon(z) =
    \begin{cases}  |z-a|^{\alpha} & \textrm{ if } \alpha < 0, \\
    \log|z-a| & \textrm{ if } \alpha =0, \\
    1 & \textrm{ if } \alpha > 0 \end{cases}
    \end{multline}
\begin{multline} \label{Ynear1}
    Y(z) =  \mathcal{O}\begin{pmatrix}  1 & 1 & \epsilon(z) \\
            1 & 1 & \epsilon(z) \\ 1 & 1 & \epsilon(z) \\
            \end{pmatrix}, \textrm{ as } z  \to 1, \\
        \textrm{where } \epsilon(z) =
    \begin{cases}  |z-1|^{\gamma} & \textrm{ if } \gamma < 0, \\
    \log|z-1| & \textrm{ if } \gamma =0, \\
    1 & \textrm{ if } \gamma > 0 \end{cases}
            \end{multline}
\begin{multline} \label{Ynear0}
    Y(z) = \mathcal{O} \begin{pmatrix}  1 & \epsilon(z) & \epsilon(z) \\
            1 & \epsilon(z) & \epsilon(z) \\ 1 & \epsilon(z) & \epsilon(z)
            \end{pmatrix}, \textrm{ as } z  \to 0,
             \\
        \textrm{where } \epsilon(z) =
    \begin{cases}  |z|^{\beta} & \textrm{ if } \beta < 0, \\
    \log|z| & \textrm{ if } \beta =0, \\
    1 & \textrm{ if } \beta > 0, \end{cases}
             \end{multline}
where the $\mathcal O$ is taken entry-wise.
    \end{itemize}

As in \cite{dei99,vgk01} one can show that there is a unique solution of the RH problem, see also \cite{kmvv04}
for the role of the endpoint conditions \eqref{Yneara}, \eqref{Ynear1} and \eqref{Ynear0}.
The first column contains the multiple orthogonal polynomials of type II with respect to
multi-indices $(n_{1},n_{2})$, $(n_{1}-1,n_{2})$ and $(n_{1},n_{2}-1)$
and the other columns contain Cauchy transforms of the polynomials times the weights.
Indeed, the solution is equal to
\begin{equation}
    Y(z) = \begin{pmatrix}
    P_{n_1,n_2}(z) & \frac{1}{2\pi i} \int_{a}^{0} \frac{P_{n_1,n_2}(x) w_1(x)}{x-z} dx &
    \frac{1}{2\pi i} \int_{0}^{1} \frac{P_{n_1,n_2}(x) w_2(x)}{x-z} dx \\
    d_1 P_{n_1-1,n_2}(z) & \frac{d_1}{2\pi i} \int_{a}^{0} \frac{P_{n_1-1,n_2}(x) w_1(x)}{x-z} dx & \frac{d_1}{2\pi i} \int_{0}^{1} \frac{P_{n_1-1,n_2}(x) w_2(x)}{x-z} dx \\
    d_2 P_{n_1,n_2-1}(z) & \frac{d_2}{2\pi i} \int_{a}^{0} \frac{P_{n_1,n_2-1}(x) w_1(x)}{x-z} dx & \frac{d_2}{2\pi i} \int_{0}^{1} \frac{P_{n_1,n_2-1}(x) w_2(x)}{x-z} dx  \end{pmatrix} \label{pn1n2fromy}
    \end{equation}
    for certain non-zero constants $d_1$ and $d_2$.
The inverse matrix $Y^{-1}$ contains multiple orthogonal polynomials of type I.

We apply the Deift-Zhou steepest descent analysis to the RH problem for $Y$ with $n_1 = n_2 = n$ in the
limit where $n \to \infty$ and $a = a_n = -1 + \frac{\sqrt{2} \tau}{n^{1/2}} \to -1$.
Via a number of transformations
\[ Y  \mapsto T \mapsto S \mapsto R \]
we arrive at a matrix valued function $R$ that tends to the identity matrix as $n \to \infty$.

Particularly relevant references on the steepest descent method for this paper are \cite{blku07,dkmvz99a,kmvv04}, see also \cite{bff09}.

\subsection{The local model RH problem}
\label{sec:localRHP}

At a crucial step in the steepest descent analysis we need to do a local analysis
at the point $0$. We have to construct there a local parametrix
that will be built out of certain special functions. In non-critical situations this
can be done with Bessel functions of order $\beta$, but in the critical regime that
we are interested in we need functions that satisfy a third order linear differential
equation. They are combined in a $3 \times 3$ matrix valued RH problem that we call the
local model RH problem and that we describe next.

\begin{figure}[t]
\begin{center}
\begin{tikzpicture}[scale = 3,line width = 1,decoration = {markings, mark = at position .5 with {\arrow{stealth};}} ]
\draw [postaction = decorate] (-1,0) -- (0,0);
\draw [postaction = decorate] (0,0) -- (1,0);
\draw [postaction = decorate] (-.707106781,-.707106781) -- (0,0);
\draw [postaction = decorate] (0,0) -- (.707106781,.707106781);
\draw [postaction = decorate] (-.707106781,.707106781) -- (0,0);
\draw [postaction = decorate] (0,0) -- (.707106781,-.707106781);
\draw (0,0) node[above]{$0$};
\draw [postaction = decorate] (.2,0) arc  (0:45:.2);
\draw (.2,0) node[above right]{$\frac{\pi}{4}$};
\draw (-1,0) node[left,font=\footnotesize]{ $\begin{pmatrix} 0 & e^{\beta \pi i} & 0 \\ -e^{\beta \pi i} & 0 & 0 \\ 0 & 0 & 1 \end{pmatrix}$};
\draw (-.707106781,.707106781) node[left,font=\footnotesize]{$\begin{pmatrix} 1 & 0 & 0 \\ e^{\beta \pi i} & 1 & 0 \\ 0 & 0 & 1 \end{pmatrix}$};
\draw (.707106781,.707106781) node[right,font=\footnotesize]{$\begin{pmatrix} 1 & 0 & 0 \\ 0 & 1 & 0 \\ 1 & 0 & 1 \end{pmatrix}$};
\draw (1,0) node[right,font=\footnotesize]{ $\begin{pmatrix} 0 & 0 & 1 \\ 0 & 1 & 0 \\ -1 & 0 & 0 \end{pmatrix}$};
\draw (.707106781,-.707106781) node[right,font=\footnotesize]{$\begin{pmatrix} 1 & 0 & 0 \\ 0 & 1 & 0 \\ 1 & 0 & 1 \end{pmatrix}$};
\draw (-.707106781,-.707106781) node[left,font=\footnotesize]{$\begin{pmatrix} 1 & 0 & 0 \\ e^{-\beta \pi i} & 1 & 0 \\ 0 & 0 & 1 \end{pmatrix}$};
\end{tikzpicture}
\caption{The contour $\Sigma_{\Psi}$ and the jump matrices of $\Psi$.}
\label{figurejumpspsi}
\end{center}
\end{figure}
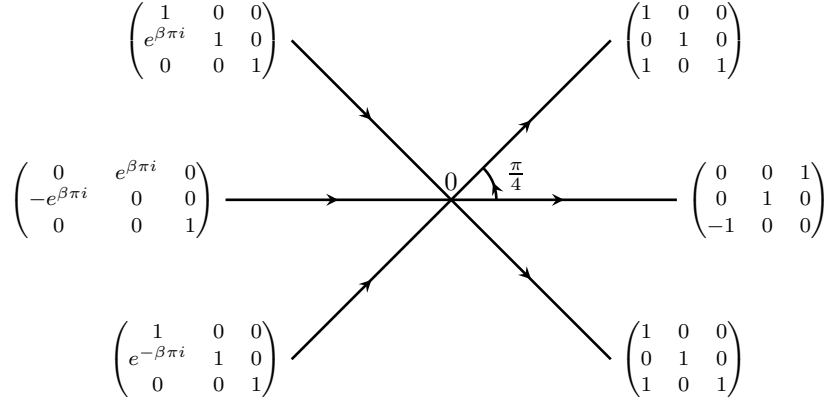

In the local model RH problem we are looking for a $3 \times 3$ matrix valued function $\Psi$ that
depends on two parameters $\beta > -1$ and $\tau \in \mathbb R$. Since $\beta$ is considered fixed,
we do not emphasize the dependence on $\beta$. We may write $\Psi(z;\tau)$ to emphasize the
dependence on $\tau$. Then $\Psi$ should satisfy the following.
\begin{itemize}
\item $\Psi$ is defined and analytic in $\mathbb{C} \setminus \Sigma_{\Psi}$ where
$\Sigma_{\Psi}$ is the contour consisting of the six oriented rays through the origin as shown
in Figure~\ref{figurejumpspsi},
\item $\Psi$ has continuous boundary values on $\Sigma_{\Psi}$ that satisfy the jump
condition
\[ \Psi_{+}(z) = \Psi_{-}(z) J_{\Psi}(z) \qquad z \in \Sigma_{\Psi}, \]
where the jump matrices $J_{\Psi}$ are also given in Figure~\ref{figurejumpspsi}.
\end{itemize}
Thus the parameter $\beta$ appears in the jump condition.

The dependence on $\tau$ is in the asymptotic condition as $z \to \infty$.
We write throughout this paper
\[ \omega = e^{2\pi i/3}. \]
\begin{itemize}
\item As $z \to \infty$ with $\pm \im z > 0$, we have
\begin{equation}
    \Psi(z) =  \sqrt{ \frac{ 2\pi}{3}} e^{\frac{\tau^{2}}{6}} z^{\frac{\beta}{3}}
    \begin{pmatrix} z^{\frac{1}{3}} & 0 & 0 \\ 0 & 1 & 0 \\ 0 & 0 & z^{-\frac{1}{3}}  \end{pmatrix}
    \Omega_{\pm}
    \left(I + \mathcal{O}\left(z^{-\frac{1}{3}}\right)\right) B_{\pm} e^{\Theta(z; \tau)}, \label{asymptoticexpansionpsi}\end{equation}
where $\Omega_{\pm}$, $B_{\pm}$ and $\Theta(z;\tau)$ are defined by
\begin{equation}  \label{definitionomega}
    \begin{aligned}
    \Omega_+ & := \begin{pmatrix} - \omega^{2} & 1 & \omega \\ 1 & -1 & -1 \\ -\omega & 1 & \omega^{2}   \end{pmatrix},
    & \quad B_+ & := \begin{pmatrix} e^{\frac{\beta\pi i}{3}} & 0 &0 \\ 0 & 1 & 0 \\ 0 & 0 & e^{-\frac{\beta \pi i}{3}} \end{pmatrix}, \\
    \Omega_- & := \begin{pmatrix} \omega & 1 & \omega^{2} \\ -1 & -1 & -1 \\ \omega^{2} & 1 & \omega \end{pmatrix},
    & \quad B_- & :=
     \begin{pmatrix} e^{-\frac{\beta \pi i}{3}} & 0 & 0 \\ 0 & 1 & 0 \\ 0 & 0 & e^{\frac{\beta \pi i}{3}} \end{pmatrix},
      \end{aligned}
\end{equation}
and
\begin{equation} \label{definitionTheta}
    \Theta(z; \tau) :=
    \begin{cases} \diag \left( \theta_{1}(z;\tau), \theta_{3}(z;\tau), \theta_{2}(z;\tau) \right) & \textrm{ for } \im z > 0, \\
    \diag \left( \theta_{2}(z;\tau), \theta_{3}(z;\tau), \theta_{1}(z;\tau) \right) & \textrm{ for } \im z < 0, \end{cases}
    \end{equation}
and the $\theta_{k}$ are defined by
\begin{equation}
    \theta_{k}(z;\tau) := -\frac{3}{2} \omega^{k} z^{\frac{2}{3}} - \tau \omega^{2k} z^{\frac{1}{3}}
    \qquad \textrm{ for } k =1,2,3. \label{definitionthetak}
\end{equation}
The expansion \eqref{asymptoticexpansionpsi} for $\Psi(z)$ as $z \to \infty$  is valid uniformly for $\tau$ in a bounded set.
\end{itemize}

We construct $\Psi(z;\tau)$ out of solutions of the third order linear differential equation
\begin{equation}
    zq'''(z) - \beta q''(z) -  \tau q'(z) + q(z) = 0.
    \label{diffeqq}
    \end{equation}
Note that this is not the same differential equation as \eqref{diffeqQ}. However, the
two are related, since if $q$ satisfies \eqref{diffeqq} then
\[ Q(z) = z^{-\beta} q''(z) \]
satisfies \eqref{diffeqQ}.

\begin{figure}[t]
\begin{center}
\begin{tikzpicture}[scale = 2,line width = 1]
\filldraw (0,0) circle (.01) node[above]{$0$};
\draw [dashed] (0,0) -- (1,0);
\draw [postaction = decorate, decoration = {markings, mark = at position .5 with {\arrowreversed{stealth};}}] (-1,0) -- (0,0);
\draw (-.5,0) node[above]{$\Gamma_{3}$};
\draw [dotted] (-1.2,0) -- (-1,0);
\draw (0,0) arc (-90:90:.5 and .3);
\draw [postaction = decorate, decoration = {markings, mark = at position .75 with {\arrow{stealth};}}] (-1,.6) -- (0,.6);
\draw [dotted] (-1.2,.6) -- (-1,.6);
\draw (-.3,.6) node[above]{$\Gamma_{2}$};
\draw (0,0) arc (90:-90:.5 and .3);
\draw [postaction = decorate, decoration = {markings, mark = at position .75 with {\arrowreversed{stealth};}}] (-1,-.6) -- (0,-.6);
\draw [dotted] (-1.2,-.6) -- (-1,-.6);
\draw (-.3,-.6) node[above] {$\Gamma_{1}$};
\end{tikzpicture}
\caption{The contours $\Gamma_{1},\Gamma_{2}$ and $\Gamma_{3}$ in the $t$-plane.
    The dashed line denotes the cut of $t^{-\beta -3}$.}
\label{figurecontoursgamma}
\end{center}
\end{figure}
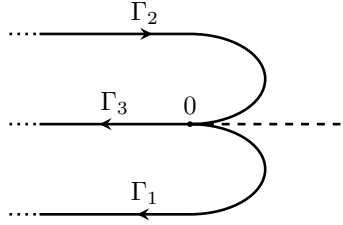

The differential equation \eqref{diffeqq} has solutions in the form of contour integrals
\begin{equation}
    q(z) = \int_{\Gamma} t^{-\beta - 3} e^{\frac{\tau}{t} - \frac{1}{2t^{2}} + zt} \,dt ,
    \end{equation}
where $\Gamma$ is an appropriate contour so that the integrand  vanishes at the endpoints of the contour $\Gamma$.
Define three contours $\Gamma_{1}, \Gamma_{2}$ and $\Gamma_{3}$ as in Figure \ref{figurecontoursgamma},
and define for $z$ with $\re z > 0$
\begin{equation}
    q_{j}(z) = \int_{\Gamma_{j}} t^{-\beta - 3} e^{\frac{\tau}{t} - \frac{1}{2t^{2}} + zt} \,dt,
    \qquad j=1,2,3,  \label{integralrepresentationsq}
    \end{equation}
where we choose the branch of $t^{-\beta -3}$ with a cut on the positive real axis, i.e.,
\[ t^{-\beta-3} = |t|^{-\beta-3} e^{(-\beta-3) i \arg t}, \qquad 0 < \arg t < 2 \pi. \]
The integrals \eqref{integralrepresentationsq} only converge for $z$ with
$\re z > 0$, but the functions $q_{j}$ can be continued analytically using contour deformations.
Branch points for the $q_{j}$-functions are $0$ and $\infty$ and we take the analytic
continuation to $\mathbb C \setminus (-\infty,0]$, thus with a branch cut on the negative real axis.

\begin{definition}
Define  $\Psi$ in the upper half plane by
\begin{equation}
    \Psi = \begin{cases}
        \begin{pmatrix} e^{2\beta \pi i} q_{1} & e^{\beta \pi i} q_{3}  & q_{2} \\
        e^{2\beta \pi i} q_{1}' & e^{\beta \pi i} q_{3}'  & q_{2}' \\
        e^{2\beta \pi i} q_{1}'' & e^{\beta \pi i} q_{3}''  & q_{2}'' \end{pmatrix} &
         0 < \arg z < \frac{\pi}{4}, \\
         \begin{pmatrix} e^{2\beta \pi i} q_{1} + q_{2} & e^{\beta \pi i} q_{3}  & q_{2} \\
         e^{2\beta \pi i} q_{1}' + q_{2}' & e^{\beta \pi i} q_{3}'  & q_{2}' \\
         e^{2\beta \pi i} q_{1}'' + q_{2}'' & e^{\beta \pi i} q_{3}''  & q_{2}'' \end{pmatrix}, &
         \frac{\pi}{4} < \arg z < \frac{3\pi}{4} \\
         \begin{pmatrix} e^{2\beta \pi i} q_{1} + q_{2} -e^{2\beta \pi i} q_{3} & e^{\beta \pi i} q_{3}  & q_{2} \\
         e^{2\beta \pi i} q_{1}' + q_{2}' -e^{2\beta \pi i} q_{3}' & e^{\beta \pi i} q_{3}'  & q_{2}' \\
         e^{2\beta \pi i} q_{1}'' + q_{2}'' -e^{2\beta \pi i} q_{3}'' & e^{\beta \pi i} q_{3}''  & q_{2}'' \end{pmatrix}, &
         \frac{3\pi}{4} < \arg z < \pi,
         \end{cases} \label{definitionpsi1}
         \end{equation}
and in the lower half plane by
\begin{equation}
    \Psi = \begin{cases}
        \begin{pmatrix} q_{2} & e^{\beta \pi i} q_{3} & -e^{2\beta \pi i} q_{1} \\
         q_{2}' & e^{\beta \pi i} q_{3}' & -e^{2\beta \pi i} q_{1}' \\
         q_{2}'' & e^{\beta \pi i} q_{3}'' & -e^{2\beta \pi i} q_{1}'' \end{pmatrix}, &
         -\frac{\pi}{4} < \arg z < 0, \\
         \begin{pmatrix} q_{2} + e^{2\beta \pi i} q_{1} & e^{\beta \pi i} q_{3} & -e^{2\beta \pi i} q_{1} \\
         q_{2}' + e^{2\beta \pi i} q_{1}' & e^{\beta \pi i} q_{3}' & -e^{2\beta \pi i} q_{1}' \\
         q_{2}'' + e^{2\beta \pi i} q_{1}''& e^{\beta \pi i} q_{3}'' & -e^{2\beta \pi i} q_{1}'' \end{pmatrix}, &
         -\frac{3\pi}{4} < \arg z <- \frac{\pi}{4} \\
         \begin{pmatrix} e^{2\beta \pi i} q_{1} + q_{2} + q_{3} & e^{\beta \pi i} q_{3} & -e^{2\beta \pi i} q_{1} \\
         e^{2\beta \pi i} q_{1}' + q_{2}' + q_{3}' & e^{\beta \pi i} q_{3}' & -e^{2\beta \pi i} q_{1}' \\
         e^{2\beta \pi i} q_{1}'' + q_{2}'' + q_{3}'' & e^{\beta \pi i} q_{3}'' & -e^{2\beta \pi i} q_{1}''  \end{pmatrix}, &
         -\pi < \arg z < - \frac{3\pi}{4}. \end{cases}
         \label{definitionpsi2}
         \end{equation}
 \end{definition}

It is then an easy exercise to check that $\Psi$ indeed satisfies the required jumps $\Psi_{+} = \Psi_{-} J_{\Psi}$
on the rays $\arg z = 0, \pm \frac{\pi}{4}, \pm \frac{3\pi}{4}$.  For the jump on the negative real axis however,
we have to take into consideration the behavior of the functions $q_{j}(z)$ as $z$ circles around $0$.  Using contour deformations one can show that for $z < 0$:
\begin{equation}
    \begin{pmatrix} q_{1,+}(z) \\ q_{2,+}(z) \\ q_{3,+}(z) \end{pmatrix} =
    \begin{pmatrix} 1 +e^{2\beta \pi i} & 1 & 0 \\ - e^{2\beta \pi i} & 0 & 0 \\ e^{2\beta \pi i} & 1 & 1 \end{pmatrix}
    \begin{pmatrix} q_{1,-}(z) \\ q_{2,-}(z) \\ q_{3,-}(z) \end{pmatrix} .\label{monodromyq}
    \end{equation}
The jump of $\Psi$ on the negative real axis follows from this in a straightforward way.

As for the asymptotic behavior, we have
\begin{proposition}
The function $\Psi$ defined in \eqref{definitionpsi1} and \eqref{definitionpsi2} satisfies the
asymptotic condition \eqref{asymptoticexpansionpsi}.
\end{proposition}

\begin{proof}
This follows from a classical steepest descent analysis applied to the contour integral representations
for the $q_{j}$ \eqref{integralrepresentationsq}.  Define the phase function $\theta(t;z,\tau)$ by
\begin{equation} \theta(t;z,\tau) := \frac{\tau}{t} - \frac{1}{2t^{2}} + zt. \end{equation}

The main contributions in the integrals occurs around the saddles $t = t_{k} = t_{k}(z;\tau)$ of $\theta$, which are the solutions to $\theta'(t) = 0$:
\begin{equation} t_{k}(z;\tau) = - \omega^{k} z^{-\frac{1}{3}} - \frac{\tau}{3} \omega^{2k} z^{-\frac{2}{3}} + \mathcal{O}\left(z^{-\frac{4}{3}}\right), \quad k = 1,2,3 .\label{tk}\end{equation}
The $\mathcal{O}$-term here is uniform for $\tau$ in compacta.  The critical values are given by
\begin{equation} \theta(t_{k}) = \theta(t_{k}; z,\tau) = -\frac{3}{2} \omega^{k} z^{\frac{2}{3}} - \tau \omega^{2k} z^{\frac{1}{3}} + \frac{\tau^{2}}{6} + \mathcal{O}\left(z^{-\frac{1}{3}}\right) .\label{thetatk}\end{equation}

We also need the second derivative of $\theta$ in the saddle points:
\begin{equation} \theta''(t_{k}) = -3 \omega^{2k} z^{\frac{4}{3}} + \mathcal{O}(z) .\label{thetadoubletk}\end{equation}

Through each saddle point $t_{k}$ there is a steepest descent path $\Upsilon_{k}$.  This is a path such that $\im \theta(t) = \im \theta(t_{k})$ for all $t \in \Upsilon_{k}$.  Let $\alpha_{k}, |\alpha_{k}| = 1$ be the tangent direction of $\Upsilon_{k}$ in $t_{k}$.  The steepest descent method then yields
\begin{equation} \int_{\Upsilon_{k}} t^{-\beta -3} e^{\theta(t)} \, dt = \alpha_{k} \sqrt{ \frac{ 2 \pi } {-\theta''(t_{k}(z))\alpha_{k}^{2}} } t_{k}(z)^{-\beta -3} e^{\theta(t_{k}(z))} \left( 1 + \mathcal{O}(z^{-1/3})\right) .\end{equation}

The fact that $\Upsilon_{k}$ is a steepest descent path guarantees that $\theta''(t_{k}(z))\alpha_{k}^{2}$ is negative.  Substituting \eqref{tk}, \eqref{thetatk} and \eqref{thetadoubletk} we find
\begin{multline*}
    \int_{\Upsilon_{k}} t^{-\beta -1} e^{\theta(t)} \, dt \\
    = \pm \sqrt{\frac{2\pi}{3}} e^{\frac{\tau^{2}}{6}} \omega^{2k} z^{-\frac{2}{3}} (-\omega^{k} z^{-\frac{1}{3}})^{-\beta -3} e^{-\frac{3}{2} \omega^{k} z^{2/3} - \tau \omega^{2k} z^{1/3} } \left(1 + \mathcal{O}\left(z^{-\frac{1}{3}}\right)\right).
    \end{multline*}

The final step in the proof is the identification of the steepest descent paths, and the deformation of the $\Gamma_{k}$ into steepest descent paths.  This gives us expressions for the $q_{k}$ in each sector.  For the second and third row of $\Psi$ we remark that by \eqref{integralrepresentationsq} differentiation of the $q_{k}$ is equivalent to increasing $\beta$ by $1$.

The final expansion for $\Psi$ then turns out to be exactly as in \eqref{asymptoticexpansionpsi}.
\end{proof}

For the further analysis we also need to know the next order term in the expansion \eqref{asymptoticexpansionpsi}.

\begin{lemma} \label{lem:Psi1}
We have as $z \to \infty$ with $\pm \im z > 0$
\begin{multline}
    \Psi(z) =  \sqrt{ \frac{ 2\pi}{3}} e^{\frac{\tau^{2}}{6}} z^{\frac{\beta}{3}}
    \begin{pmatrix} z^{\frac{1}{3}} & 0 & 0 \\ 0 & 1 & 0 \\ 0 & 0 & z^{-\frac{1}{3}}  \end{pmatrix}
    \Omega_{\pm} \\
    \times \left(I + (\Psi_1)_{\pm} z^{-\frac{1}{3}} + \mathcal{O}\left(z^{-\frac{2}{3}}\right)\right)
    B_{\pm} e^{\Theta(z)}, \label{refinedexpansion}
\end{multline} where the constant matrices
$\left(\Psi_{1}\right)_{\pm}$ are given by
\begin{align}
    \left(\Psi_{1} \right)_{+} & = - \frac{\tau}{3}\left(\frac{\tau^{2}}{9} + \beta +1 \right)
    \begin{pmatrix} \omega & 0 & 0 \\ 0 & 1 & 0 \\ 0 & 0 & \omega^{2} \end{pmatrix}
    - \frac{\tau}{9} \begin{pmatrix} 0 & \omega^{2} - \omega & 1- \omega \\ \omega^{2} - 1 & 0 & 1- \omega \\ 1- \omega^{2} & \omega^{2} - \omega & 0               \end{pmatrix} \\
    \left(\Psi_{1}\right)_- & = -\frac{\tau}{3} \left( \frac{\tau^{2}}{9} + \beta +1\right)
    \begin{pmatrix} \omega^{2} & 0 & 0 \\ 0 & 1 & 0 \\ 0 & 0 & \omega \end{pmatrix}
    - \frac{\tau}{9} \begin{pmatrix} 0 & \omega^{2} - \omega & \omega^{2} - 1 \\ 1- \omega & 0 & 1- \omega^{2} \\ \omega- 1& \omega- \omega^{2} & 0             \end{pmatrix}
\end{align}
\end{lemma}

\begin{proof}
The first row of $\Psi_{1}$ can be found in a straightforward way by expressing that
the entries in the first row of $\Psi(z)$ solve the differential equation \eqref{diffeqq}.
For the second and third rows we  increase $\beta$ in \eqref{diffeqq} by $1$ and $2$, respectively.
\end{proof}

\section{Second tool: modified equilibrium problem and Riemann surface}
\label{sec:modeq}

\subsection{Modified equilibrium problem}
One of the transformations in the Deift/Zhou steepest descent analysis of the RH problem for
(multiple) orthogonal polynomials is typically based on the
limiting zero distribution of the associated polynomials. In our situation these are
given by the vector equilibrium problem from Definition \ref{vectorequilibrium}.
As explained in Subsection \ref{subsectionphasetransition}, the endpoint $b = b_{a}$
of one of the supports is varying with $a$ and tends to $0$ as $a \to -1$.  Working with
measures with varying supports around $0$ would cause major technical problems.
Therefore, following \cite{blku07, clku06}, we use a modified equilibrium problem, where
the positivity of the measures is not required. So we will be dealing with
signed measures.

\begin{definition}
The modified equilibrium problem asks for two signed measures $\mu_{1}$ and $\mu_{2}$ minimizing
the energy \eqref{energy} among all signed measures with $\supp(\mu_1) \subset [a,0]$, $\supp(\mu_2) \subset [0,1]$
and  $\int d\mu_{1} = \frac{1}{2}$, $\int d\mu_{2} = \frac{1}{2}$.
\end{definition}

The modified equilibrium measures $\mu_{1}$ and $\mu_{2}$ are unique, and have full supports $[a,0]$ and $[0,1]$,
respectively. In fact, the modification of the equilibrium problem comes down to forcing the equilibrium measures to
have full supports, at the expense of losing positivity of one of the measures near $0$.

In the symmetric case $a =-1$ the modified equilibrium measures coincide with the usual equilibrium measures.
In this case the equilibrium densities $\psi_{1}$ and $\psi_{2}$ are positive on the
whole supports, and around $0$ they blow up like an inverse cube root.

In the general case $a \neq -1$ the equilibrium density on the smaller interval is  positive on the full interval.
The density on the larger interval becomes negative in an interval between $0$ and $x_0$ where $x_0$ depends on $a$
in such a way that
\begin{equation}
x_{0} = x_{0}(a) = \frac{ (a+1)^{3}}{108} + \mathcal{O}(a+1)^{4} \quad \textrm{ as } a \to -1,
\end{equation}
see \eqref{x0expansion} below.
For $a \neq -1$ both densities $\psi_{1}(x)$ and $\psi_{2}(x)$ behave like $x^{-2/3}$ as $x \to 0$.
A sketch of the densities is given in Figure~\ref{figuremodified}.

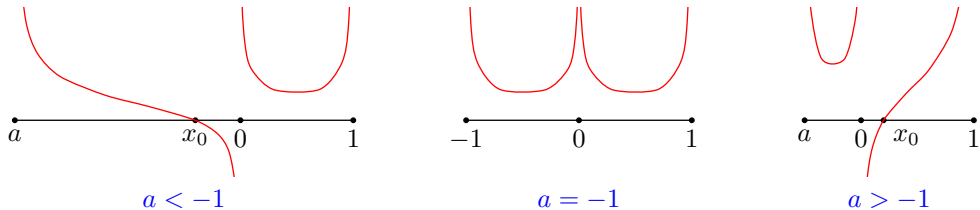
\begin{figure}[t]
\centering
\begin{tikzpicture}[scale = 1.5,  line width = .5]
\useasboundingbox (-5,-1) rectangle (4.5,1);
\draw (2.75,-0.7) node[text = blue] {$a > -1$};
\draw (0,-0.7) node[text = blue] {$a = -1$};
\draw (-3.5,-0.7) node[text = blue] {$a < -1$};
\clip (-7,-.5) rectangle (5.5,1);
\draw [color = black] (-1,0) -- (1,0);
\filldraw (-1,0) circle (.02) node[below] {$-1$};
\filldraw (0,0) circle (.02) node[below] {$0$};
\filldraw (1,0) circle (.02)  node[below] {$1$};
\draw[color = red, line width = .5] plot[smooth] coordinates{
(-.99, 1.933780056)
(-.96, 0.886107499)
(-.90, 0.5128480254)
(-.70, 0.2731187430)
(-.30, 0.2697788644)
(-.10, 0.4495865000)
(-.04, 0.6779079288)
(-.01, 1.192546280)
(-.001, 2.8)};
\draw[color = red, line width = .5] plot[smooth] coordinates{
(.99, 1.933780056)
(.96, 0.886107499)
(.90, 0.5128480254)
(.70, 0.2731187430)
(.30, 0.2697788644)
(.10, 0.4495865000)
(.04, 0.6779079288)
(.01, 1.192546280)
(.001, 2.8)};
\draw[black] (-5,0) -- (-2,0);
\filldraw (-5,0) circle (.02) node[below] {$a$};
\filldraw (-3,0) circle (.02) node[below] {$0$};
\filldraw (-2,0) circle (.02)  node[below] {$1$};
\filldraw (-3.4,0) circle (.02) node[below]{$x_{0}$};
\draw[color = red, line width = .5] plot[smooth] coordinates{
(-2.01, 1.933780056)
(-2.04, 0.886107499)
(-2.10, 0.5128480254)
(-2.30, 0.2731187430)
(-2.70, 0.2697788644)
(-2.90, 0.4495865000)
(-2.96, 0.6779079288)
(-2.99, 1.192546280)
(-2.999, 2.8)
};
\draw[color = red, line width = .5] plot[smooth] coordinates{
(-4.98, 1.98825)
(-4.96,1.3975)
(-4.90,0.86765)
(-4.80, 0.59365)
(-4.60,0.38920)
(-4.20,0.22264)
(-3.80, 0.1189)
(-3.40, 0)
(-3.20,-0.122315)
(-3.10,-0.28348)
(-3.04, -0.61694)
(-3.02, -1.0285)
};
\draw[black] (2,0) -- (3.5,0);
\filldraw (2,0) circle (.02) node[below] {$a$};
\filldraw (2.5,0) circle (.02) node[below] {$0$};
\filldraw (3.5,0) circle (.02)  node[below] {$1$};
\filldraw (2.7,0) circle (.02) node[below right] {$x_{0}$};
\draw (2.75,-1) node[text = blue] {$a > -1$};
\draw[color = red, line width = .5] plot[smooth] coordinates{
(2.005, 3.867560112)
(2.02,1.772214998)
(2.05,1.025696051)
(2.15,0.5462374860)
(2.35,0.5395577288)
(2.45, 0.8991730000)
(2.48,1.355815858)
(2.495,2.385092560)};
\draw[color = red, line width = .5] plot[smooth] coordinates{
(3.49,3.9765)
(3.48,2.795)
(3.45,1.7153)
(3.40,1.1873)
(3.30,0.7784)
(3.10,0.44528)
(2.90,0.2378)
(2.70,0)
(2.60,-0.24463)
(2.55,-0.56696)
(2.52,-1.23388)
(2.51,-2.057)
};
\end{tikzpicture}
\caption{A sketch of the modified equilibrium densities in the three cases $a < -1$, $a=1$ and $a> -1$.
The modified equilibrium density becomes negative near $0$ on the larger of the two intervals.}
\label{figuremodified}
\end{figure}

Define the logarithmic potentials $U^{\mu}$ of a (signed) measure $\mu$ by
\begin{equation}
    U^{\mu}(x) = \int \log \frac{1}{|x-y|} \, d\mu(y), \quad x \in \mathbb{C} .\label{defpotentials}
    \end{equation}
The Euler-Lagrange variational conditions \cite{dei99,sato97} for $\mu_{1}$ and $\mu_{2}$
then say that there exist constants $l_{1}, l_{2} \in \mathbb{R}$ such that
\begin{equation}
    \begin{aligned}
        2U^{\mu_{1}}(x) + U^{\mu_{2}}(x) & = l_{1},  \textrm{ for } x \in [a,0], \\
    U^{\mu_{1}}(x) + 2 U^{\mu_{2}}(x) & = l_{2}, \textrm{ for } x \in [0,1],
    \end{aligned}
    \label{variationalconditions}
    \end{equation}
and these conditions characterize the modified equilibrium measures.
For the non-modified equilibrium measures we would have an inequality instead of
equality for $x$ in the gap of the supports.  The fact that for $\mu_{1}$ and $\mu_{2}$
the Euler-Lagrange variational conditions have such a simple form on the full intervals $[a,0]$
and $[0,1]$ will be important for the further analysis.

\subsection{Riemann surface}

The modified equilibrium problem is easiest to analyze by means
of an appropriate three-sheeted Riemann surface $\mathcal{R}$.
Define $\mathcal{R}$ by taking three copies of the Riemann sphere $\overline{\mathbb{C}}$ with cuts
\begin{equation}
    \mathcal{R}_{0} := \overline{\mathbb{C}} \setminus  [a,1],
    \quad \mathcal{R}_{1} :=  \overline{\mathbb{C}} \setminus [a,0],
    \quad \mathcal{R}_{2} :=  \overline{\mathbb{C}} \setminus [0,1] ,
    \end{equation}
and gluing them together along these cuts in the usual crosswise manner, see Figure \ref{riemannsurface}.

\begin{figure}[t]
\begin{center}
\begin{tikzpicture}[scale = 2,line width = .5]
\draw [very thick] (-1.5,1) -- (1,1);
\filldraw (-1.5,1) circle (.02);
\filldraw (0,1) circle (.02);
\filldraw (1,1) circle (.02);
\draw (-1.5,1) node[above]{$a$};
\draw (0,1) node[above]{$0$};
\draw (1,1) node[above]{$1$};
\filldraw (-1.5,0) circle (.02);
\filldraw (0, -1) circle (.02);
\filldraw (0,0) circle (.02);
\filldraw (1,-1) circle (.02);
\draw [very thick](-1.5,0) -- (0,0);
\draw [very thick](0,-1) -- (1,-1);
\draw (-1.5,1.3) -- (1.7,1.3);
\draw (-2,.7) -- (1.2,.7);
\draw (1.7,1.3) -- node[below right]{ $\mathcal{R}_{0}$}(1.2,.7);
\draw (-1.5,1.3) -- (-2,.7);
\draw (-1.5,0.3) -- (1.7,0.3);
\draw (-2,-.3) -- (1.2,-.3);
\draw (1.7,.3)-- node[below right]{ $\mathcal{R}_{1}$}(1.2,-.3);
\draw (-1.5,.3) -- (-2,-.3);
\draw (-1.5,-.7) -- (1.7,-.7);
\draw (-2,-1.3) -- (1.2,-1.3);
\draw (1.7,-.7)-- node[below right]{ $\mathcal{R}_{2}$}(1.2,-1.3);
\draw (-1.5,-.7) -- (-2,-1.3);
\draw [dashed] (-1.5,1) -- (-1.5,0);
\draw [dashed] (0,1) -- (0,-1);
\draw [dashed] (1,1) -- (1,-1);
\end{tikzpicture}
\caption{The sheets of the Riemann surface $\mathcal{R}$}
\label{riemannsurface}
\end{center}
\end{figure}
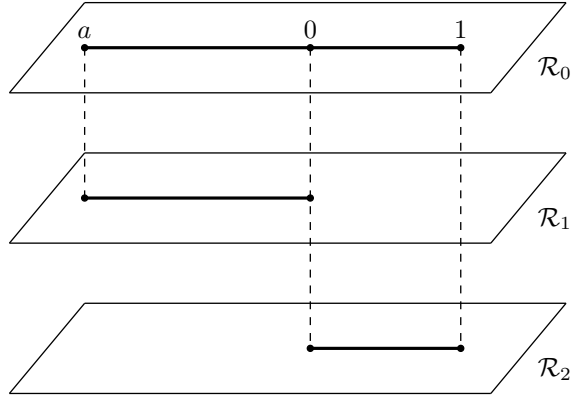

The Riemann surface has genus zero and can be defined by the polynomial equation
\begin{equation}
        4 a \xi^3 - 2 (a+1) z \xi^3 - 3 (a-1) z \xi^2 +  (a-1) z = 0.
    \label{polequation} \end{equation}
Solving for $z$, we find a rational function
\begin{equation}
     z = z(\xi) = \frac{4 a \xi^3}{2(a+1) \xi^3 + 3 (a-1) \xi^2 - (a-1)},
    \label{definitionf}
    \end{equation}
which defines a conformal map from $\xi \in \overline{\mathbb C}$ to $z \in \mathcal R$,
so that the branch points $a$, $0$, and $1$ of $\mathcal R$ correspond to $\xi=-1, 0, 1$,
respectively.
The restriction of the inverse mapping of \eqref{definitionf} to the sheet $\mathcal{R}_{i}$ is denoted by $\xi_{i}$
\[ \xi_i : \mathcal R_i \to \mathbb C. \]

The $\xi$-functions map the sheets of $\mathcal{R}$ to certain domains
\[ \widetilde{\mathcal R}_i := \xi_i(\mathcal R_i) \]
 of $\overline{\mathbb{C}}$.
 Then $\widetilde{\mathcal{R}}_{1}$ and $\widetilde{\mathcal{R}}_{2}$ are bounded,
 while $\widetilde{\mathcal{R}}_{0}$ is unbounded.  We use $\gamma_{1}^{\pm}$ and $\gamma_{2}^{\pm}$
 to denote the arcs bounding $\widetilde{\mathcal{R}}_{1}$ and $\widetilde{\mathcal{R}}_{2}$
 with clockwise orientation as in Figure~\ref{imagesxi} below.

\subsection{Properties of the modified equilibrium problem}

We can make use of the Riemann surface to prove the following
properties of the modified equilibrium measures. We will not
give all details in the following calculations.

The modified equilibrium measures are described in terms of the algebraic equation
\begin{equation} \label{zetaequation}
    \zeta^3  -\frac{3z-2z^*(a) - 1 - a}{4z(z-a)(z-1)} \zeta - \frac{z-z^*(a)}{4z^2(z-a)(z-1)} = 0
    \end{equation}
where $z^* = z^*(a)$ is a certain solution of
\begin{equation} \label{zstarequation}
    64 (z^*)^3 - 48(a+1) (z^*)^2 - (15a^2-78a+15) z^* - (a+1)^3 = 0.
    \end{equation}
It can be shown that \eqref{zstarequation} has three distinct real
solutions if $a < 0$. We use $z^*(a)$ to denote the middle one of
the three solutions and this is the value that is used in \eqref{zetaequation}.

\begin{proposition} \label{prop:psi12}
For $a< 0$, the following hold.
\begin{enumerate}
\item[\rm (a)] The three solutions of \eqref{zetaequation} are given by
\begin{equation}  \label{zeta012}
    \begin{aligned}
    \zeta_0(z) & = \int_a^0 \frac{\psi_1(x)}{z-x} d x + \int_0^1 \frac{\psi_2(x)}{z-x} dx, \\
    \zeta_1(z) & = - \int_a^0 \frac{\psi_1(x)}{z-x} dx, \\
    \zeta_2(z) & = - \int_0^1 \frac{\psi_2(x)}{z-x} dx,
    \end{aligned}
    \end{equation}
where $\psi_1$ and $\psi_2$ are the densities of the modified equilibrium measures.
\item[\rm (b)] The densities satisfy
\begin{equation} \label{psi12}
    \begin{aligned}
    \psi_1(x) & = \frac{1}{2\pi i} \left(\zeta_{1,+}(x) - \zeta_{1,-}(x)\right), \quad x \in (a,0), \\
    \psi_2(x) & = \frac{1}{2\pi i} \left(\zeta_{2,+}(x) - \zeta_{2,-}(x)\right), \quad x \in (0,1).
    \end{aligned}
    \end{equation}
\item[\rm (c)] There is $x_0 = x_0(a) \in (a,1)$ with the same sign as $a+1$ such that
\begin{itemize}
\item if $-1 < a < 0$, then $\psi_1(x) > 0$ for $x \in (a,0)$ and $\psi_2(x) < 0$ if and only if $0 < x <  x_0$,
\item if $a < -1$, then $\psi_2(x) > 0$ for $x \in (0,1)$ and $\psi_1(x) < 0$ if and only if $x_0 < x < 0$.
\end{itemize}
See also Figure \ref{figuremodified}.
\item[\rm (d)] We have
\begin{equation}
    \label{x0expansion}
    x_0(a) = \frac{(a+1)^3}{108} + \mathcal O\left(a+1\right)^{4} \qquad \text{as } a \to -1.
    \end{equation}
\end{enumerate}
\end{proposition}

\begin{proof}

Let $\zeta_j$, $j=0,1,2$ be defined by \eqref{zeta012}, so that we clearly
have
\begin{equation} \label{elemzeta1}
    \zeta_0(z) + \zeta_1(z) + \zeta_2(z) = 0
    \end{equation}
It follows from the variational conditions \eqref{variationalconditions} that
$\zeta_{0,+} = \zeta_{1,-}$ on $(a,0)$ and $\zeta_{0,-} = \zeta_{2,-}$ on $(0,1)$.

Thus if we consider $\zeta_j$ as a function defined on the sheet $\mathcal R_j$
for $j=0,1,2$, then this function extends to a meromorphic function on $\mathcal R$.
Since (due to  the normalization $\int d\mu_1 = \int d\mu_2 = 1/2$)
\begin{equation}
  \zeta_0(z) = z^{-1} + \mathcal O(z^{-2}), \qquad
    \zeta_j(z) = - \frac{1}{2} z^{-1} + \mathcal O(z^{-2}),
    \end{equation}
as $z \to \infty$, the meromorphic function has simple zeros at the three
points at infinity. There are simple poles at $a$ and $1$ and a possible double pole at $0$.
In addition there is a fourth simple zero at a point $z^*$.

Then the product $\zeta_0 \zeta_1 \zeta_2$ is a rational function in the complex
plane with a zero at $z^*$, simple poles at $a$, $-1$, a double pole at $0$, and it
behaves as $\frac{1}{4} z^{-3}$ as $z \to \infty$. This means that
\begin{equation} \label{elemzeta3}
    \zeta_0(z) \zeta_1(z) \zeta_2(z) = \frac{z-z^*}{4z^2(z-a)(z-1)}.
    \end{equation}
Similar considerations show that
\begin{equation} \label{elemzeta2}
    \zeta_0(z) \zeta_1(z) + \zeta_0(z) \zeta_2(z) + \zeta_1(z) \zeta_2(z) = -\frac{3z-q}{4z(z-a)(z-1)}
    \end{equation}
for some $q$.   Thus $\zeta_j$, $j=1,2,3$ are the three solutions of the algebraic
equation
\begin{equation} \label{algebraicequation}
    \zeta^3  -\frac{3z-q}{4z(z-a)(z-1)} \zeta - \frac{z-z^*}{4z^2(z-a)(z-1)} = 0
    \end{equation}
Inserting $\zeta = \zeta_1(z) = -1/(2z) + \mathcal O(z^{-2})$ into \eqref{algebraicequation}
shows that $q = 2z^* + a + 1$, which gives us the equation \eqref{zetaequation}.

The discriminant of \eqref{zetaequation} with respect to $\zeta$  has the form
\[ \frac{Q_{2}(z)}{16z^{4}(z-a)^{3}(z-1)^{3}} \]
where $Q_2(z)$ is a certain quadratic polynomial in $z$ that we calculated with Maple.
The poles $a$, $0$ and $1$ of the discriminant correspond to the branch points of the Riemann surface.
The quadratic polynomial should have a double zero, since otherwise there would be more branch points.
This leads to a condition on $z^*$, which turns out to be given by \eqref{zstarequation}.
Again we made these calculations with Maple.  This proves part (a) of the proposition.

The relevant solution $z^*$ of \eqref{zstarequation} is the one that is $0$ for $a=-1$. This solution is then
well-defined as a real analytic function for $a \in (-\infty,0)$. We have
\begin{equation} \label{zstarat-1}
        z^*(a) = - \frac{(a+1)^3}{108}  + \mathcal{O} \left((a+1)^4\right) \qquad \text{as } a \to -1,
        \end{equation}
which can be obtained from \eqref{zstarequation}.
The double root of $Q_2(z)$ turns out to be equal to
\[ x_0(a) = \frac{ (1+a)^3 + ( 6a^2 -42 a + 6) z^*(a) -15(1+a)(z^*(a))^{2} + 8(z^*(a))^3}
        {18 \left( 1 - a + a^2  -2(1+a)z^*(a) +(z^*(a))^{2} \right)},
        \]
which can be shown to also satisfy a cubic equation
\begin{multline} \label{x0a}
    (27a^2-46a+27) x_0^3 -3(a+1)(9a^2-14a+9) x_0^2 \\
    + 3a(11a^2-14a+11) x_0 - a(a+1)^3 = 0.
    \end{multline}
There are three real distinct solutions of \eqref{x0a} if $a<0$ and $x_0(a)$ is the middle one.
The expansion \eqref{x0expansion} follows from \eqref{x0a} and part (d) follows.

Part (b) follows immediately from part (a) and the Sokhotskii-Plemelj formulas
that tell us how to recover the density of a measure from its Cauchy transform.

Finally, to prove part (c), we suppose that $-1 < a <0$. It can then be shown
from the above formulas (it is not immediate, however) that $0 < x_0(a) < 1$.
Since $x_0(a)$ is a zero of the discriminant, the cubic equation \eqref{zetaequation}
has a double solution if $z = x_0(a)$. Since $0 < x_0(a) < 1$, we have $\xi_{0,+}(x_0(a)) = \overline{\xi_{2,+}(x_0(a))}$
and $\xi_1(x_0(a))$ is real. Thus we have $\xi_{0,+}(x_0(a)) = \xi_{2,+}(x_0(a))$, which means
since $\xi_{0,+} = \xi_{2,-}$ on $(0,1)$, that $\psi_2$ vanishes at $x_0(a)$ by \eqref{psi12}.
Since $x_0(a)$ is the only zero of the discriminant in $(a,1)$, it also follows that $x_0(a)$ is
the only zero of $\psi_2$, and that $\psi_1$ has no zeros.
Thus $\psi_1 > 0$ on $(a,0)$. It is a consequence of the fact that the point $0$ of the Riemann surface
is a double pole, that $\psi_1$ and $\psi_2$ have opposite signs near $0$. Thus $\psi_2(x) < 0$
for $0 < x < x_0(a)$ and part (c) of the proposition is proved in case $-1 < a < 0$.

The proof for $a < -1$ is similar.
\end{proof}

\section{Steepest descent analysis of the RH problem}
\label{sec:sda}

\subsection{First transformation}
\label{sectionfirsttransfo}

We start from the RH problem for $Y$ with $n_1 = n_2 =n$. We also take $a < 0$ close to $-1$
but for the moment it is arbitrary and fixed. We use the modified equilibrium measures $\mu_1$ and $\mu_2$
that are supported on the two intervals $[a,0]$ and $[0,1]$ respectively.

Define $g$-functions by
\begin{equation} \label{definitiongj}
    g_j(z)  =  \int \log(z-s) d\mu_{j}(s), \qquad j=1,2,
    \end{equation}
where we use the main branch of the logarithm.
Hence $g_{1}$ is defined with a branch cut on $(-\infty, 0]$ and  $g_{2}$ with a branch cut on $(-\infty,1]$.
The boundary values of the $g$-functions along the real axis are given by
\begin{equation} \label{boundaryvaluesg}
    \begin{aligned}
    g_{1,\pm}(x) & = - U^{\mu_{1}}(x) \pm \pi i \int_{x}^{0} d\mu_{1}(s), \\
    g_{2,\pm}(x) & = -U^{\mu_{2}}(x) \pm \pi i \int_{x}^{1} d\mu_{2}(s),
    \end{aligned}
    \end{equation}
    were $U^{\mu_{1}}$ and $U^{\mu_{2}}$ are the logarithmic potentials \eqref{defpotentials}.
From \eqref{variationalconditions} and \eqref{boundaryvaluesg} we obtain
\begin{equation} \label{boundaryvaluesg2}
    \begin{aligned}
    g_{1,+}(x) + g_{1,-}(x) + g_{2,\pm}(x) & = - l_1 \pm \frac{1}{2} \pi i, && x \in [a,0], \\
  g_1(x) + g_{2,+}(x) + g_{2,-}(x) & = -l_2, && \quad x \in [0,1],
  \end{aligned}
  \end{equation}

Now define the first transformation $Y \mapsto T$ as
\begin{multline} \label{YtoT}
    T(z)  = \begin{pmatrix} 1 & 0 & 0 \\ 0 & e^{-2n(l_{1} + \tfrac{1}{2} \pi i)} & 0 \\ 0 & 0 &  e^{-2nl_{2}} \end{pmatrix} Y(z) \\
        \times  \begin{pmatrix} e^{-2n(g_{1}(z) + g_{2}(z)) } & 0 & 0 \\ 0 & e^{2n(g_{1}(z)+l_{1} + \tfrac{1}{2} \pi i)} & 0 \\ 0 & 0 & e^{2n(g_{2}(z) + l_{2})} \end{pmatrix}.
\end{multline}

This transformation normalizes the RH problem at $\infty$, since the $g$-functions behave like
$\frac{1}{2} \log z + O(z^{-1})$ as $z \to \infty$. Thus $T(z) = I + \mathcal{O}(z^{-1})$ as $z \to \infty$.
The jumps for $T$ are conveniently expressed in terms of the two functions
$\varphi_{j}$, $j = 1,2$ defined by
\begin{equation}
    \begin{aligned}
    \varphi_{1}(z) & :=  -2g_{1}(z) - g_{2}(z) - l_{1} +
    \begin{cases} \tfrac{1}{2} \pi i & \text{ for } \im z > 0, \\
        - \tfrac{1}{2} \pi i & \text{ for } \im z < 0,
        \end{cases} \\
    \varphi_{2}(z) & := -2g_{2}(z) - g_{1}(z) - l_{2} +
    \begin{cases} \pi i & \text{ for } \im z > 0, \\
        -  \pi i & \text{ for } \im z < 0.
        \end{cases}
    \end{aligned}
    \label{definitionvarphi} \end{equation}
Then by \eqref{jY}, \eqref{YtoT}, and \eqref{boundaryvaluesg2} one obtains the jump matrices
\begin{align}
    J_{T}(x) & = \begin{pmatrix} e^{2n\varphi_{1,+}(x)} & w_{1}(x) & 0 \\ 0 & e^{2n\varphi_{1,-}(x)} & 0 \\ 0 & 0 & 1
        \end{pmatrix} \textrm{ for } x \in (a,0) \label{jT1}, \\
     J_T(x) &  = \begin{pmatrix} e^{2n\varphi_{2,+}(x)} & 0 & w_{2}(x) \\ 0 & 1 & 0 \\ 0 & 0 & e^{2n\varphi_{2,-}(x)}
     \end{pmatrix}  \textrm{ for } x \in (0,1) \label{jT2}.
     \end{align}

Thus $T$ satisfies the following RH problem.
\begin{itemize}
\item $T$ is an analytic $3\times 3$ matrix valued function on $\mathbb{C} \setminus [a,1]$,
\item $T$ satisfies jump conditions $T_{+}(x) = T_{-}(x)J_{T}(x)$ for $x \in (a,0) \cup (0,1)$, with $J_{T}$ given by \eqref{jT1} and \eqref{jT2},
\item $T$ is normalized at infinity:
\begin{equation}
    T(z) = I + \mathcal{O}\left(z^{-1} \right) \textrm{ as } z \to \infty, \end{equation}
\item near $a$, $0$ and $1$ the function $T$ has the same behavior as $Y$, see \eqref{Yneara}, \eqref{Ynear1}, and \eqref{Ynear0}.
\end{itemize}

\subsection{Second transformation: opening of the lenses}
\label{sectionlenses}

The functions $e^{2n\varphi_{k,\pm}}, k = 1,2$ appearing in the jump matrices $J_{T}$ in \eqref{jT1} and \eqref{jT2}
are rapidly oscillating for large $n$, since the boundary values $\varphi_{k,\pm}$ are purely imaginary.  One may easily check that
\begin{align}
    \varphi_{1,\pm }(x) & = \mp 2\pi i \int_{x}^{0} d\mu_{1}(s) \quad \textrm{ for } x \in (a,0),  \label{boundaryvaluesphi1} \\
    \varphi_{2,\pm}(x) & = \pm 2 \pi i \int_{0}^{x} d\mu_{2}(s) \quad \textrm{ for } x \in (0,1). \label{boundaryvaluesphi2}
    \end{align}

The oscillations are turned into exponential decay by the so-called
opening of the lenses. Choose smooth paths $\Sigma_{1}^{\pm}$ connecting
$a$ and $0$ with $\Sigma_{1}^{+}$ in the upper half-plane and $\Sigma_1^-$
in the lower half plane. Similarly, choose paths $\Sigma_{2}^{+}$ and $\Sigma_{2}^{-}$ connecting $0$ and $1$.
Define
\begin{equation} \label{SigmaS}
    \Sigma_S := \Sigma_{1}^{\pm} \cup \Sigma_2^{\pm} \cup [a,1].
      \end{equation}
The intervals $[a,0]$, $[0,1]$ and the paths $\Sigma_{i}^{\pm}$ define $4$ bounded regions
that are referred to as the lenses around $[a,0]$ and $[0,1]$, see Figure~\ref{figuresigmas}.

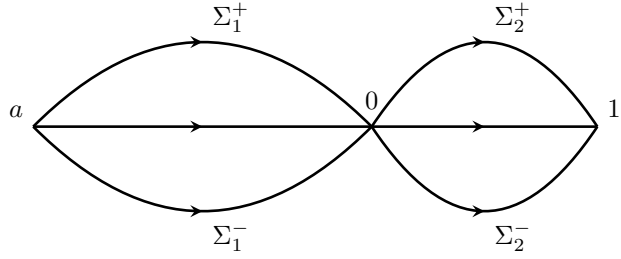
\begin{figure}[t]
\centering
\begin{tikzpicture}[scale = 3,line width = 1, decoration = {markings, mark = at position .5 with {\arrow{stealth};}} ]
\draw [postaction = {decorate}] (-1.5,0) node[above left ]{$a$} -- (0,0) node[above=3pt]{0};
\draw [postaction = {decorate}](0,0) -- (1,0) node[above right]{1} ;
\draw [postaction = {decorate}](-1.5,0) .. controls (-1,.5) and (-.5,.5) .. node[above right]{$\Sigma_{1}^{+}$}(0,0);
\draw [postaction = {decorate}](-1.5,0) .. controls (-1,-.5) and (-.5,-.5) .. node[below right]{$\Sigma_{1}^{-}$} (0,0);
\draw [postaction = {decorate}](0,0) .. controls (.3333,.5) and (.6666,.5) .. node[above right]{$\Sigma_{2}^{+}$}(1,0);
\draw [postaction = {decorate}](0,0) .. controls (.3333,-.5) and (.6666,-.5) .. node[below right]{$\Sigma_{2}^{-}$}(1,0);
\end{tikzpicture}
\caption{The contour $\Sigma_S$ and the lenses around $[a,0]$ and $[0,1]$.}
\label{figuresigmas}
\end{figure}

For $j=1,2$, let $V_j$ be a simply connected neighborhood of $\Delta_j$ such that
the analytic factor $h_{j}$ in the weight function $w_j$ is analytic and non-zero in $V_j$.
We assume that $\Sigma^{\pm}_j \subset V_j$ for $j=1,2$.
Then $w_j$ has an analytic continuation from $\Delta_j$ to $V_j$ with some cuts, which we also denote by $w_{j}$:
\begin{equation} \label{w1w2continuation}
  \begin{aligned}
   w_{1}(z) & = (z-a)^{\alpha} (-z)^{\beta} h_{1}(z), &
   \textrm{ for } z \in V_{1} \setminus \left((-\infty, a] \cup [0,\infty) \right), \\
   w_{2}(z) & = z^{\beta} (1-z)^{\gamma} h_{2}(z), &
   \textrm{ for } z \in V_{2} \setminus \left((-\infty, 0] \cup [1,\infty) \right).
   \end{aligned}
\end{equation}

Then, following \cite[section 4]{kmvv04}, we define the next transformation $T \mapsto S$ by
\begin{equation}
    S(z) = \begin{cases} T(z) \begin{pmatrix} 1 & 0 & 0 \\
    - w_{1}(z)^{-1} e^{2n\varphi_{1}(z)} & 1 & 0 \\ 0 & 0 & 1 \end{pmatrix}, &
    \begin{array}{l} z \textrm{ in the upper part of} \\
        \textrm{the lens around } [a,0],
        \end{array} \\
    T(z)\begin{pmatrix} 1 & 0 & 0 \\
    w_{1}(z)^{-1} e^{2n\varphi_{1}(z)} & 1 & 0 \\ 0 & 0 & 1 \end{pmatrix}, &
    \begin{array}{l} z \textrm{ in the lower part of} \\
        \textrm{the lens around } [a,0],
        \end{array}  \end{cases}  \label{TtoS1} \end{equation}
\begin{equation}
    S(z) = \begin{cases} T(z) \begin{pmatrix} 1 & 0 & 0 \\ 0 & 1 & 0 \\
    - w_{2}(z)^{-1} e^{2n\varphi_{2}(z)} & 0 & 1 \end{pmatrix}, &
        \begin{array}{l} z \textrm{ in the upper part of } \\
        \textrm{ the lens around } [0,1],
        \end{array} \\
    T(z)\begin{pmatrix} 1 & 0 & 0 \\ 0 & 1 & 0 \\
    w_{2}(z)^{-1} e^{2n\varphi_{2}(z)} & 0 & 1 \end{pmatrix}, &
        \begin{array}{l} z \textrm{ in the lower part of } \\
        \textrm{ the lens around } [0,1],
        \end{array} \end{cases}   \label{TtoS2} \end{equation}
and
\begin{equation}
    S(z) = T(z) \quad \textrm{elsewhere}.
    \end{equation}

It is clear that this transformation does not affect the behavior at infinity.
The jump matrix $J_S$ for $S$ on the intervals $(a,0)$ and $(0,1)$ are
\begin{equation}
    J_{S}(x)  = \begin{cases}
        \begin{pmatrix} 0 & w_{1}(x) & 0 \\- w_{1}(x)^{-1} & 0 & 0 \\ 0 & 0 & 1 \end{pmatrix},
        & x \in (a,0), \\
       \begin{pmatrix} 0 & 0 & w_{2}(x) \\ 0 & 1 & 0 \\ -w_{2}(x)^{-1} & 0 & 0  \end{pmatrix},
      &  x \in (0,1).
       \end{cases} \label{jS1}
\end{equation}
The transformation has introduced jumps on $\Sigma_j^{\pm}$ which are
\begin{equation}
    J_{S}(z) = \begin{cases}
        \begin{pmatrix} 1 & 0 & 0 \\ w_{1}(z)^{-1} e^{2n\varphi_{1}(z)} & 1 & 0 \\ 0 & 0 & 1 \end{pmatrix},
        & z \in \Sigma_{1}^{\pm}, \\
        \begin{pmatrix} 1 & 0 & 0 \\ 0 & 1 & 0 \\ w_{2}(z)^{-1} e^{2n\varphi_{2}(z)} & 0 & 1 \end{pmatrix},
        & z \in \Sigma_{2}^{\pm}.
        \end{cases} \label{jS2} \end{equation}

Finally, the behavior near $a$, $0$ and $1$ changes because of the factors $w_{j}^{-1}$
in the transformation \eqref{TtoS1}-\eqref{TtoS2}.  The Riemann-Hilbert problem for $S$ then reads:
\begin{itemize}
\item $S$ is analytic on $\mathbb{C} \setminus \Sigma_{S}$,
\item $S$ has jumps $S_{+} = S_{-}J_{S}$, where $J_{S}$ is given by \eqref{jS1} and \eqref{jS2},
\item $S(z) = I + \mathcal{O}\left(z^{-1}\right)$ as $z \to \infty$,
\item near the endpoints of the intervals $S$ behaves as
\begin{equation}
    S(z) = \mathcal{O} \begin{pmatrix} \epsilon_{1}(z) & \epsilon_{2}(z) & 1 \\
    \epsilon_{1}(z) & \epsilon_{2}(z) & 1 \\ \epsilon_{1}(z) & \epsilon_{2}(z) & 1  \end{pmatrix},
    \quad \textrm{ as } z  \to a, \label{Sneara} \end{equation}
where
\[ \epsilon_{1}(z), \epsilon_{2}(z) = \begin{cases}
    1, |z-a|^{\alpha} & \textrm{ if } \alpha< 0, \\
    \log|z-a|, \log|z -a| & \textrm{ if } \alpha =0, \\
    |z-a|^{-\alpha}, 1 & \textrm{ if } \alpha > 0, z \textrm{ inside the lens}, \\
    1 , 1 & \textrm{ if } \alpha > 0, z \textrm{ outside the lens},
    \end{cases} \]
\begin{equation}
    S(z) = \mathcal{O} \begin{pmatrix} \epsilon_{1}(z) & 1 & \epsilon_{2}(z) \\
    \epsilon_{1}(z) & 1 & \epsilon_{2}(z) \\
    \epsilon_{1}(z) & 1 & \epsilon_{2}(z) \end{pmatrix},
    \quad \textrm{ as } z  \to 1, \label{Snear1} \end{equation}
where
\[ \epsilon_{1}(z), \epsilon_{2}(z) = \begin{cases}
        1, |z-1|^{\gamma} & \textrm{ if } \gamma< 0,  \\
            \log|z-1|,\log|z-1| & \textrm{ if } \gamma =0, \\
            |z-1|^{-\gamma},1 & \textrm{ if } \gamma > 0, z \textrm{ inside the lens}, \\
            1 , 1 & \textrm{ if } \gamma > 0, z \textrm{ outside the lens}, \end{cases} \]
\begin{equation}
    S(z) = \mathcal{O}\begin{pmatrix} \epsilon_{1}(z) & \epsilon_{2}(z) & \epsilon_{2}(z) \\
        \epsilon_{1}(z) & \epsilon_{2}(z) & \epsilon_{2}(z) \\
        \epsilon_{1}(z) & \epsilon_{2}(z) & \epsilon_{2}(z)  \end{pmatrix},
        \quad \textrm{ as } z  \to 0, \label{Snear0} \end{equation}
where
\[  \epsilon_{1}(z),\epsilon_{2}(z) = \begin{cases}
            1, |z|^{\beta} & \textrm{ if } \beta< 0, \\
            \log|z|, \log|z| & \textrm{ if } \beta =0, \\
            |z|^{-\beta} , 1 & \textrm{ if } \beta > 0, z \textrm{ inside the lenses}, \\
            1,1 & \textrm{ if } \beta > 0, z\textrm{ outside the lenses.}
            \end{cases}  \]
\end{itemize}

For later analysis it will be important to know how $\re \varphi_1$ and $\re \varphi_2$
behave on the lips of the lenses. From \eqref{jS2} we see that we would like to have
\begin{equation} \label{phijineq}
    \re \varphi_j(z) < 0 \quad \text{for } z \in \Sigma_j^{\pm},
\end{equation}
for $j=1,2$. The inequality \eqref{phijineq} will indeed hold if $\mu_j$
is a positive measure, and this can be proven using the Cauchy-Riemann equations.
Recall however, that $\mu_1$ and $\mu_2$ are signed measures. The inequality
\eqref{phijineq} will be violated for $z$ on the parts of $\Sigma_j$ that are close
to the interval where $\mu_j$ is negative. By Proposition \eqref{prop:psi12} we
have that $\mu_1$ is negative on $(x_0(a), 0)$ if $a < -1$, and
that $\mu_2$ is negative near $(0, x_0(a))$ if $-1 < a < 0$, where
$x_0(a) = \mathcal{O}((a+1)^3)$ as $a \to -1$.

We write $\varphi_j(z;a)$ to emphasize the dependence on $a$.

\begin{lemma} \label{lem:revarphij}
There exist positive constants $C_0$ and $C_1$, independent of $a$,
such that for every $a$ sufficiently close to $-1$, we have
\begin{equation} \label{revarphij}
    \re \varphi_j(z;a) \leq C_0 |a+1| |z|^{1/3} - C_1 |z|^{2/3},
    \quad z \in \Sigma_j^{\pm}, \, |z| < 1/2.
       \end{equation}
for $j=1,2$,
\end{lemma}

\begin{proof}
We have by the definitions \eqref{zeta012}, \eqref{definitiongj}, and \eqref{definitionvarphi} that
$\varphi_j' = \zeta_j - \zeta_0$ for $j=1,2$. The constant in \eqref{definitionvarphi} is taken
so that $\varphi_j(0)$ which means that
\begin{equation} \label{varphijintegral}
    \varphi_j(z;a) = \int_0^z (\zeta_j(s;a) - \zeta_0(s;a)) ds
    \end{equation}
where $\zeta_0, \zeta_1, \zeta_2$ are the three solution of the cubic equation \eqref{zetaequation},
where we emphasize the dependence on $a$.

As $s \to 0$ with $\im s > 0$ we can compute from \eqref{zetaequation} that
\begin{align} \label{zeta012at0}
    \zeta_0(s;a) & =
        c_0 \omega^2 s^{-2/3} + c_1 \omega s^{-1/3} + \mathcal{O}(1), \\
    \zeta_1(s;a) & = c_0 s^{-2/3} + c_1 s^{-1/3} + \mathcal{O}(1), \\
    \zeta_2(s;a) & = c_0 \omega s^{-2/3} + c_1 \omega^2 s^{-1/3} + \mathcal{O}(1)
    \end{align}
     uniformly for $a$ close to $-1$,
    with real constants $c_0 = (-z^*(a)/(4a))^{1/3}$ and $c_1 = - (2z^*(a) + a + 1)/(12 a c_0)$.
     Because of the behavior \eqref{zstarat-1} of $z^*(a)$ we have that
\begin{equation} \label{c0c1at-1}
    \begin{aligned}
    c_0 = c_0(a) & = - \tfrac{2^{2/3}}{12} (a+1) + \mathcal{O}((a+1)^2) \\
  c_1 = c_1(a) & =  - \tfrac{2^{1/3}}{2}  + \mathcal{O}(a+1)
  \end{aligned}
  \end{equation}
  as $a \to -1$.

 Using \eqref{zeta012at0} in \eqref{varphijintegral}  we find
 \begin{equation} \label{varphi2z}
 \varphi_2(z;a) =  3^{3/2} i c_0(a) z^{1/3} - \tfrac{1}{2} 3^{3/2}   i c_1(a) z^{2/3} + \mathcal{O}(z)
 \end{equation}
 as $z \to 0$ with $\im z > 0$.  By \eqref{c0c1at-1} we have that $c_1(a)$ tends to a negative constant as $a \to -1$.
 Since we may assume that the lens is opened with a positive angle at $0$, we find that
\[ \re \left(- \frac{3}{2} \sqrt{3} i c_1(a) z^{2/3} \right) \leq - C_1 |z|^{2/3}, \qquad z \in \Sigma_2^{+} \]
for some constant $C_1 > 0$ independent of $a$. Using this in \eqref{varphi2z} we obtain
\eqref{varphijintegral} for $j=2$ and $z \in \Sigma_2^+$ in a fixed size neighborhood of $z=0$,
say $|z| < r_0$.
The  inequality \eqref{varphijintegral} then also holds for $|z| < 1/2$ (maybe with different constant $C_1$), since
$\varphi_2(z;a) \to \varphi_2(z;-1)$ as $a \to -1$ uniformly for $r_0 \leq |z| \leq 1/2$,   and
$\re \varphi_2(z;-1) < -C_3 < 0$ for $z \in \Sigma_2^+$, $r_0 \leq |z| \leq 1/2$,  and some $C_3 > 0$.

The inequality \eqref{varphijintegral} for $j=2$ and $z \in \Sigma_2^-$ and for $j=1$ follow in a similar way.
\end{proof}

It follows from \eqref{revarphij} that we indeed have that
$\re \varphi_j(z;a) < 0$ for $z \in \Sigma_j^{\pm}$, except for $z$ in a small exceptional
neighborhood of $z=0$, whose radius shrinks as $\mathcal{O}\left((a+1)^3 \right)$ as $a \to-1$.

\subsection{Outer parametrix}
\label{sectionouter}

The next step is to construct an approximation to $S$. This so-called parametrix consists
of an outer parametrix $N$ that gives an approximation away from the endpoints $a$, $0$ and $1$
and local parametrices $P$ around each of the endpoints.

\subsubsection{Riemann-Hilbert problem for $N$}
The outer parametrix $N$ should satisfy
\begin{itemize}
\item $N$ is analytic on $\mathbb{C} \setminus  [a,1]$,
\item $N$ satisfies the jump conditions
\begin{equation} N_{+}(x) = \left\{\begin{array}{ccc} N_{-}(x)\left(\begin{array}{ccc} 0 & w_{1}(x) & 0 \\ -w_{1}(x)^{-1} & 0 & 0 \\ 0 & 0 & 1 \end{array}\right) & \textrm{ if } x \in (a,0), \\ N_{-}(x) \left(\begin{array}{ccc} 0 & 0 & w_{2}(x) \\ 0 & 1 & 0 \\ - w_{2}(x)^{-1} & 0 & 0 \end{array}\right) & \textrm{ if } x \in (0,1),  \end{array} \right. \label{jumpsn} \end{equation}
\item $N$ is normalized at infinity:
\begin{equation}
N(z) = I + \mathcal{O}\left( \frac{1}{z}\right) \textrm{ as } z \to \infty. \label{asymptoticconditionn}
\end{equation}
\end{itemize}

\subsubsection{Solution in a special case} \label{sectionSpecialN}
First we will find a solution $\widetilde{N}$ to this  problem for the case that both weights $w_{1}$ and $w_{2}$
are identically $1$ on their respective intervals.
We solve the problem for $\widetilde{N}$ by using the  Riemann surface $\mathcal{R}$ introduced before.
There  is a similar construction in \cite{blku04} and so we do not go into much detail here.

The function $\xi = \xi_j(z)$ maps the sheet $\mathcal R_j$ of the Riemann surface
onto the domain $\widetilde{\mathcal R}_j$ as shown in Figure \ref{imagesxi}
that are separated by two closed contours $\gamma_1$ and $\gamma_2$ that we orient
in the clockwise direction.
We write
\[ p_j = p_j(a) = \xi_j(\infty), \qquad j = 0,1,2. \]

\begin{figure}[t]
\begin{center}
\begin{tikzpicture}[scale = 10, line width = 2,decoration = {markings, mark = at position .10 with {\arrowreversed[black]{stealth};},mark = at position .325 with {\arrowreversed[black]{stealth};},mark = at position .6 with {\arrow[black]{stealth};},mark = at position .82 with {\arrow[black]{stealth};}}]
\definecolor{vgray}{gray}{.8};
\draw[postaction = {decorate},color = black, fill=vgray] plot[smooth cycle] coordinates{
(0,0) 
(0.08590953109, -0.1335884204)
(.1148060708, -.1675132670) 
(.1581949169, -.2084176381) 
(.1945044293, -.2348288530) 
(.2275727664, -.2536513070) 
(.2887164544, -.2772033557) 
(.3462617751, -.2875451803) 
(.4018033470, -.2872880819) 
(.4560833681, -.2770127689) 
(.5095061476, -.2558924319) 
(.5623148769, -.2211114187) 
(.5885401974, -.1965206258) 
(.6146662821, -.1644524465) 
(.6407050855, -.1190351095) 
(.6536949634, -0.0851242451)
(.6666666666, 0)
(.6536949634, 0.0851242451)
(.6407050855, .1190351095)
(.6146662821, .1644524465)
(.5885401974, .1965206258)
(.5623148769, .2211114187)
(.5095061476, .2558924319)
(.4560833681, .2770127689)
(.4018033470, .2872880819)
(.3462617751, .2875451803)
(.2887164544, .2772033557)
(.2275727664, .2536513070)
(.1945044293, .2348288530)
(.1581949169, .2084176381)
(.1148060708, .1675132670)
(0.08590953109,0.1335884204)
(0,0)
(-0.08590953109,-0.1335884204)
(-.1148060708, -.1675132670)
(-.1581949169, -.2084176381)
(-.1945044293, -.2348288530)
(-.2275727664, -.2536513070)
(-.2887164544, -.2772033557)
(-.3462617751, -.2875451803)
(-.4018033470, -.2872880819)
(-.4560833681, -.2770127689)
(-.5095061476, -.2558924319)
(-.5623148769, -.2211114187)
(-.5885401974, -.1965206258)
(-.6146662821, -.1644524465)
(-.6407050855, -.1190351095)
(-.6536949634, -0.0851242451)
(-.6666666666, 0)
(-.6536949634, 0.0851242451)
(-.6407050855, .1190351095)
(-.6146662821, .1644524465)
(-.5885401974, .1965206258)
(-.5623148769, .2211114187)
(-.5095061476, .2558924319)
(-.4560833681, .2770127689)
(-.4018033470, .2872880819)
(-.3462617751, .2875451803)
(-.2887164544, .2772033557)
(-.2275727664, .2536513070)
(-.1945044293, .2348288530)
(-.1581949169, .2084176381)
(-.1148060708, .1675132670)
(-0.08590953109, 0.1335884204)
};
\filldraw (0,0) circle (.005) node[above=2mm]{$0$};
\filldraw (-2/3,0) circle (.005) node[left]{$-1$}; 
\filldraw (2/3,0) circle (.005) node[right]{$1$}; 
\filldraw (-.434,0) circle (.005) node[above]{$p_{1}$};
\filldraw (.334,0) circle (.005) node[above]{$p_{2}$};
\draw (.4018033470, -.2872880819) node[below]{$\gamma_{2}^{-}$};
\draw (-.4018033470, -.2872880819) node[below]{$\gamma_{1}^{-}$};
\draw (-.4018033470, .2872880819) node[above]{$\gamma_{1}^{+}$};
\draw (.4018033470, .2872880819) node[above]{$\gamma_{2}^{+}$};
\draw (0,.3) node{$\widetilde{\mathcal{R}}_{0}$};
\draw (-.3,-.1) node{$\widetilde{\mathcal{R}}_{1}$};
\draw (.25,-.1) node{$\widetilde{\mathcal{R}}_{2}$};
\end{tikzpicture}
\caption{The images of the sheets of $\mathcal{R}$ under the functions $\xi_j$.
The location of $p_{1} = \xi_1(\infty)$ and $p_{2} = \xi_{2}(\infty)$ depend on the choice of $a$,
but the contours $\gamma_{1}^{\pm}$ and $\gamma_{2}^{\pm}$ do not.}
\label{imagesxi}
\end{center}
\end{figure}
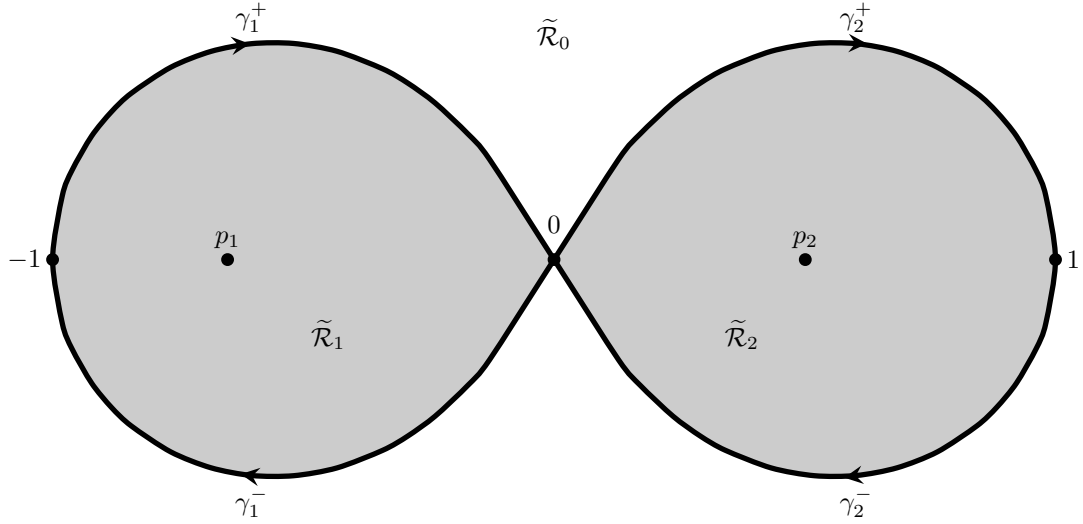

A solution  $\widetilde{N}$ is given in the form
\begin{equation}
    \widetilde{N}(z) = \begin{pmatrix} \widetilde{N}_{0}(\xi_{0}(z)) & \widetilde{N}_{0}(\xi_{1}(z)) & \widetilde{N}_{0}(\xi_{2}(z)) \\ \widetilde{N}_{1}(\xi_{0}(z)) & \widetilde{N}_{1}(\xi_{1}(z)) & \widetilde{N}_{1}(\xi_{2}(z)) \\ \widetilde{N}_{2}(\xi_{0}(z)) & \widetilde{N}_{1}(\xi_{2}(z)) & \widetilde{N}_{2}(\xi_{2}(z)) \end{pmatrix}, \label{definitionn}
\end{equation}
with the following functions $\widetilde{N}_{j}$
that are analytic on $\overline{\mathbb{C}} \setminus (\gamma_1^{\pm} \cup \gamma_{2}^{\pm})$
\begin{equation}
    \widetilde{N}_j(\xi)  =  \frac{p_j(p_j^2 - 1)^{1/2}}{\prod\limits_{\substack{i=0 \\ i \neq j}}^2 (p_j -p_i)}
        \frac{1}{\xi (\xi^2 - 1)^{1/2}} \prod\limits_{\substack{i=0 \\ i \neq j}}^2 (\xi - p_i)
\label{definitionnk}
\end{equation}
with appropriate modifications if $p_0 = \infty$ (which happens if $a=-1$).
The branch cut of the square roots $(\xi^2 - 1)^{1/2}$ and $(p_j^2-1)^{1/2}$ in \eqref{definitionnk} is defined along
$\gamma_1^+ \cup \gamma_2^+$, see Figure \ref{imagesxi}.
It can then be checked that the matrix function $\widetilde{N}$ defined by \eqref{definitionn}--\eqref{definitionnk}
indeed satisfies the conditions in the Riemann-Hilbert problem for $\widetilde{N}$.

\subsubsection{Solution in general case}
Now we turn to the problem for $N$, with jumps involving $w_{1}$ and $w_{2}$.
This can be solved using analogues to the Szeg\H{o} function as  in \cite{kmw09,kmvv04}.
We look for three functions $D_0$, $D_{1}$ and $D_{2}$
satisfying
\begin{enumerate}
\item[(a)] $D_0$ is analytic and non-zero in $\overline{\mathbb C} \setminus [a,1]$,
\item[(b)] $D_1$ is analytic and non-zero in $\overline{\mathbb C} \setminus [a,0]$,
\item[(c)] $D_2$ is analytic and non-zero in $\overline{\mathbb C} \setminus [0,1]$,
\item[(d)] $D_0$, $D_1$ and $D_2$ have limiting values on $(a,0)$ and $(0,1)$ such that
\[ \frac{D_{1,+}}{D_{0,-}} = \frac{D_{1,-}}{D_{0,+}} = w_1 \qquad \text{on } (a,0) \]
and
\[ \frac{D_{2,+}}{D_{0,-}} = \frac{D_{2,-}}{D_{0,+}} = w_2 \qquad \text{on } (0,1). \]
\end{enumerate}

Having $D_0$, $D_1$ and $D_2$ we define $N$ by
\begin{equation}
    N(z) := \begin{pmatrix} D_0(\infty)^{-1} & 0 & 0 \\ 0 & D_1(\infty)^{-1} & 0 \\ 0 & 0 & D_2(\infty)^{-1} \end{pmatrix}
        \widetilde{N}(z) \begin{pmatrix} D_{0}(z) & 0 & 0 \\ 0 & D_{1}(z) & 0 \\ 0 & 0 & D_{2}(z) \end{pmatrix}.
        \label{factorizationn} \end{equation}
We will abbreviate this as
\begin{equation}
N(z) = D_{\infty}^{-1} \widetilde{N}(z) D(z) \label{ddinfty},
\end{equation}
and one can check that $N$ indeed solves the Riemann-Hilbert problem for $N$.

From the jump properties of $D_0$, $D_1$ and $D_2$ it follows that
\[ (D_0 D_1 D_2)_+ = (D_0 D_1 D_2)_- \qquad \text{on } (a,0) \text{ and } (0,1). \]
Thus $D_0 D_1 D_2$ is analytic across these cuts. We also make sure that the possible
singularities at $a$, $0$ and $1$ are removable. Then $D_0 D_1 D_2$ is a constant
and we can choose a normalization such that
\[ D_0 D_1 D_2 \equiv 1. \]

\subsubsection{Szeg\H{o} functions}
In order to find $D_0$, $D_1$ and $D_2$ we write
\begin{equation}
    D_j(z)  = \mathcal D(\xi_j(z)), \qquad j=0,1,2,
\end{equation}
for some yet to be determined function $\mathcal D$ on the $\xi$-Riemann sphere.
Recall that $\xi_0$, $\xi_1$ and $\xi_2$ are the mapping functions from the respective
sheets of the Riemann surface $\mathcal R$ to the Riemann sphere.

Then $\mathcal D$ has to satisfy
\begin{enumerate}
\item[(a)] $\mathcal D : \overline{\mathbb{C}} \setminus (\gamma_1 \cup \gamma_2) \to \mathbb C $
is analytic and non-zero.
\item[(b)] On $\gamma_1 \cup \gamma_2$ there is a jump
    \begin{equation}
     \mathcal D_+(\xi) = w_j(z) \mathcal D_-(\xi),  \qquad  \xi \in \gamma_j, \quad j = 1,2,
        \end{equation}
        where $z = z(\xi)$ is related to $\xi$ by \eqref{definitionf}.
\end{enumerate}

Then by taking logarithms we get
    \begin{equation}
     \log \mathcal  D_+(\xi) = \log w_j(z) + \log \mathcal D_-(\xi),  \qquad  \xi \in \gamma_j^+, \qquad j=1,2,
        \end{equation}
which by the Sokhotskii Plemelj formula is solved by the Cauchy transforms
\[ \log \mathcal D(\xi) = \frac{1}{2\pi i}
    \left(\int_{\gamma_1} \frac{\log w_1(z(s))}{s-\xi} ds + \int_{\gamma_2} \frac{\log w_2(z(s))}{s-\xi} ds \right) + C_1 \]
where $C_1$ is an arbitrary constant.
Thus
\[ \mathcal D(\xi) = C \exp \left[ \frac{1}{2\pi i}
    \left(\int_{\gamma_1} \frac{\log w_1(z(s))}{s-\xi} ds + \int_{\gamma_2} \frac{\log w_2(z(s))}{s-\xi} ds \right)\right] \]
with $C = e^{C_1}$ and
\begin{equation}  \label{expressiond}
    D_j(z) = C \exp \left[ \frac{1}{2\pi i}
        \left(\int_{\gamma_1} \frac{\log w_1(z(s))}{s-\xi_j(z)} ds + \int_{\gamma_2} \frac{\log w_2(z(s))}{s-\xi_j(z)} ds \right) \right],
        \end{equation}
    for $j=0,1,2$. The constant $C$ can be taken so that $D_0 D_1 D_2 \equiv 1$.
This completes the construction of $D_0$, $D_1$ and $D_2$ and therefore of $N$.

\begin{example}
In the case where $h_1\equiv 1$, $h_2 \equiv 1$, we can evaluate
the Szeg\H{o} functions explicitly.
Indeed, we find for
\begin{equation} \label{w1w2special}
\begin{aligned}
    w_1(x) = w_1^{(\alpha,\beta)}(x) & = (x-a)^{\alpha} (-x)^{\beta}, &&  x \in (a,0), \\
    w_2(x) = w_2^{(\beta, \gamma)}(x) & = x^{\beta} (1-x)^{\gamma}, && x \in (0,1),
    \end{aligned}
    \end{equation}
    that
\begin{equation} \label{DfunctionsABC}
\begin{aligned}
    D_0(z)  & = C \left(\frac{\xi_0(z) - \xi_1(\infty)}{\xi_0(z) - \xi_1(a)} \right)^{\alpha}
        \left( \frac{\xi_0(z) - \xi_2(\infty)}{\xi_0(z) - \xi_2(1)} \right)^{\gamma}  \\
        & \qquad \times
        \left( \frac{(\xi_0(z) - \xi_1(\infty))(\xi_0(z) - \xi_2(\infty))}{\xi_0(z)^2} \right)^{\beta}, \\
    D_1(z)  & = C \left( (z-a)\frac{\xi_1(z) - \xi_1(\infty)}{\xi_1(z) - \xi_1(a)} \right)^{\alpha}
        \left( \frac{\xi_1(z) - \xi_2(\infty)}{\xi_1(z) - \xi_2(1)} \right)^{\gamma} \\
        & \qquad \times \left( (-z) \frac{(\xi_1(z) - \xi_1(\infty))(\xi_1(z) - \xi_2(\infty))}{\xi_1(z)^2} \right)^{\beta}, \\
    D_2(z) & = C\left(\frac{\xi_2(z) - \xi_1(\infty)}{\xi_2(z) - \xi_1(a)} \right)^{\alpha}
        \left( (1-z) \frac{\xi_2(z) - \xi_2(\infty)}{\xi_2(z) - \xi_2(1)} \right)^{\gamma} \\
        & \qquad \times
        \left( z \frac{(\xi_2(z) - \xi_1(\infty))(\xi_2(z) - \xi_2(\infty))}{\xi_2(z)^2} \right)^{\beta}.
    \end{aligned}
\end{equation}
with appropriate choice of branches for the exponents.
\end{example}

\subsubsection{Behavior of Szeg\H{o} functions near $0$}

From \eqref{polequation} and the choice of branches $\xi_0$, $\xi_1$, $\xi_2$,
we obtain
\begin{equation}
\begin{aligned}
    \xi_0(z) & = - C(a) \omega^{\mp} z^{1/3} + \mathcal{O}(z),  \\
  \xi_1(z) & = - C(a) z^{1/3} + \mathcal{O}(z), \\
  \xi_2(z) & = - C(a) \omega^{\pm} z^{1/3} + \mathcal{O}(z),
  \end{aligned} \qquad \textrm{for } \pm \im z > 0
  \end{equation}
with
\[  C(a) = \left(\frac{a-1}{4a}\right)^{1/3} > 0. \]

We use this in \eqref{DfunctionsABC} together with $\xi_1(a) = -1$, $\xi_2(1) = 1$, $\xi_1(\infty) = p_1(a)$,
$\xi_2(\infty) = p_2(a)$, to obtain the leading behavior of the Szeg\H{o} functions at $0$
for the case $h_1 \equiv 1$, $h_2 \equiv 1$. It follows from \eqref{DfunctionsABC} that
\begin{equation} \label{Dnear0noh}
\begin{aligned}
  D_0(z) & = C \left( - p_1(a)\right)^{\alpha} p_2(a)^{\gamma}
                    \left(\frac{-p_1(a) p_2(a)}{C(a)^2} \right)^{\beta} e^{\pm \beta \pi i /3} z^{-2 \beta/3}
                    (1 + \mathcal{O}(z^{1/3})), \\
    D_1(z) & = C \left(a p_1(a) \right)^{\alpha} p_2(a)^{\gamma}
                     \left(\frac{- p_1(a) p_2(a)}{C(a)^2}\right)^{\beta} z^{\beta/3} (1 + \mathcal{O}(z^{1/3})),
            \\
  D_2(z) & = C \left(-p_1(a) \right)^{\alpha} p_2(a)^{\gamma}
                    \left(\frac{-p_1(a) p_2(a)}{C(a)^2} \right)^{\beta} e^{\mp \beta \pi i/3} z^{\beta/3}
                    (1 + \mathcal{O}(z^{1/3}))
\end{aligned}
\end{equation}
as $z \to 0$ with $\pm \im z > 0$.

The effect of the analytic factors $h_1$ and $h_2$ comes in the form of contour integrals
\[ \frac{1}{2\pi i} \oint_{\gamma_j}  \frac{\log h_j(z(s))}{s-\xi} ds, \qquad j = 1,2, \qquad \xi \in \mathbb C \setminus \gamma_j, \]
see \eqref{expressiond}.
Because of analyticity we can deform $\gamma_j$ to a contour $\gamma_j^{\epsilon}$ in the region $\widetilde{\mathcal R}_j$,
which leaves the integral unchanged if $\xi \in \mathbb C \setminus \widetilde{\mathcal R}_j$
and picks up a residue contribution of $\log(h_j(z(\xi))$ in case $\xi \in \widetilde{\mathcal R}_j$
is close to $\gamma_j$, in particular if $\xi$ is close to $0$.
In this way we find the following behavior as $\xi \to 0$,
\begin{align} \label{hjintegral}
         \frac{1}{2\pi i} \oint_{\gamma_j}  \frac{\log h_j(z(s))}{s-\xi} ds
            =
            \begin{cases}
            c_j + \mathcal O(\xi), &  \xi \in \mathbb C \setminus \widetilde{\mathcal R}_j \\
            c_j - \log h_j(0) + \mathcal O(\xi)), &  \xi \in \widetilde{\mathcal R}_j
            \end{cases}
\end{align}
where
\[ c_j = \frac{1}{2\pi i} \oint_{\gamma_j^{\epsilon}}  \frac{\log h_j(z(s))}{s} ds.  \]
By the change of variables $z(s) = x$, $s = \xi_j(z)$, we turn this integral into an integral on a counter
that circles around $\Delta_j$ in counterclockwise direction. Bringing this integral to $\Delta_j$ we obtain
\begin{equation} \label{formulacj}
    c_j =  \frac{1}{2\pi i} \int_{\Delta_j} \log h_j(x)
    \left( \left(\frac{\zeta_j'}{\zeta_j} \right)_- - \left(\frac{\zeta_j'}{\zeta_j} \right)_+ \right)(x) dx.
    \end{equation}

Combining \eqref{Dnear0noh}, \eqref{hjintegral}, \eqref{formulacj} we find that
for general analytic factors we have
\begin{equation} \label{Dnear0}
\begin{aligned}
  D_0(z) & = C e^{c_1 + c_2} \left( - p_1(a)\right)^{\alpha} p_2(a)^{\gamma}
                    \left(\frac{-p_1(a) p_2(a)}{C(a)^2} \right)^{\beta} e^{\pm \beta \pi i /3} z^{-2 \beta/3}
                    (1 + \mathcal{O}(z^{1/3})), \\
    D_1(z) & = C h_1(0) e^{c_1 + c_2} \left(a p_1(a) \right)^{\alpha} p_2(a)^{\gamma}
                     \left(\frac{- p_1(a) p_2(a)}{C(a)^2}\right)^{\beta} z^{\beta/3} (1 + \mathcal{O}(z^{1/3})),
            \\
  D_2(z) & = C h_2(0) e^{c_1 + c_2} \left(-p_1(a) \right)^{\alpha} p_2(a)^{\gamma}
                    \left(\frac{-p_1(a) p_2(a)}{C(a)^2} \right)^{\beta} e^{\mp \beta \pi i/3} z^{\beta/3}
                    (1 + \mathcal{O}(z^{1/3}))
\end{aligned}
\end{equation}
as $z \to 0$ with $\pm \im z > 0$, with $c_1$ and $c_2$ given by \eqref{formulacj}.

For $a=-1$ we have $-p_1(a) = p_2(a) = \frac{1}{\sqrt{3}}$ and $C(a) = 2^{-1/3}$.
Since all quantities depend analytically on $a$, we find from \eqref{Dnear0}
\begin{equation} \label{Dnear0aisminus1}
\begin{aligned}
  D_0(z;a) & = C e^{c_1 + c_2} \frac{2^{\frac{2}{3} \beta}}{3^{\frac{1}{2}(\alpha+\gamma) + \beta}}
    e^{\pm \beta \pi i /3} z^{-2 \beta/3}  (1 + \mathcal{O}(z^{1/3}) + \mathcal O(a+1)), \\
    D_1(z;a) & = C h_1(0) e^{c_1 + c_2} \frac{2^{\frac{2}{3} \beta}}{3^{\frac{1}{2}(\alpha+\gamma) + \beta}}
        z^{\beta/3} (1 + \mathcal{O}(z^{1/3}) + \mathcal O(a+1)),
            \\
  D_2(z;a) & = C h_2(0) e^{c_1 + c_2} \frac{2^{\frac{2}{3} \beta}}{3^{\frac{1}{2}(\alpha+\gamma) + \beta}}
                    e^{\mp \beta \pi i/3} z^{\beta/3}
                    (1 + \mathcal{O}(z^{1/3}) + \mathcal O(a+1))
\end{aligned}
\end{equation}
as $z \to 0$ and $a \to -1$.

\subsubsection{Behavior of $N$ around the endpoints}

For the further analysis of the Riemann-Hilbert problem for $S$ we need to know the behavior
of $N$ around the endpoints $a,0$ and $1$.  To that end we also need to know
how the Szeg\H{o}-functions $D_0$, $D_1$ and $D_2$
behave around these points.

The functions $D_{0,1,2}(z;w_1,w_2)$ are multiplicative in $w_1$ and $w_2$.
 Then we can split off the analytical factors $h_{1}, h_{2}$ from the weights $w_{1}, w_{2}$
 and write
 \[ D_j(z; w_1, w_2) = D_j(z; w_1^{(\alpha, \beta)}, w_2^{(\beta, \gamma)}) D_j(z;h_1, h_2) \]
 where the functions $D_j(z; w_1^{(\alpha, \beta)}, w_2^{(\beta, \gamma)})$
  associated with the weights \eqref{w1w2special} are given in \eqref{DfunctionsABC} above.

The explicit expressions allow  us to prove the following proposition:

\begin{proposition}
Around the branch points $N$ has the following behavior:
\begin{equation}   \label{behaviorN}
  N(z) =\begin{cases}
        \mathcal{O}\begin{pmatrix}
        (z-a)^{-\frac{1+ 2\alpha }{4}} & (z-a)^{\frac{2\alpha - 1}{4}} & 1 \\
        (z-a)^{-\frac{1+ 2\alpha }{4}} & (z-a)^{\frac{2\alpha - 1}{4}} & 1 \\
        (z-a)^{-\frac{1+ 2\alpha }{4}} & (z-a)^{\frac{2\alpha - 1}{4}} & 1
        \end{pmatrix} &  \textrm{ as } z \to a,  \\
    \mathcal{O}\begin{pmatrix}  (z-1)^{-\frac{1 + 2\gamma}{4}}  & 1 & (z-1)^{\frac{2\gamma -1}{4}} \\
    (z-1)^{-\frac{1 + 2\gamma}{4}}  & 1 & (z-1)^{\frac{2\gamma -1}{4}} \\
    (z-1)^{-\frac{1 + 2\gamma}{4}}  & 1 & (z-1)^{\frac{2\gamma -1}{4}}
    \end{pmatrix} & \textrm{ as } z \to 1, \\
    \mathcal{O}\begin{pmatrix} z^{-\frac{2\beta + 1}{3}} & z^{\frac{\beta - 1}{3}} & z^{\frac{\beta - 1}{3}} \\
    z^{-\frac{2\beta + 1}{3}} & z^{\frac{\beta - 1}{3}} & z^{\frac{\beta - 1}{3}} \\
    z^{-\frac{2\beta + 1}{3}} & z^{\frac{\beta - 1}{3}} & z^{\frac{\beta - 1}{3}}
    \end{pmatrix} & \textrm{ as } z \to 0.
    \end{cases} \end{equation}
\label{propbehaviorN}
\end{proposition}
\begin{proof}
From the expressions \eqref{definitionn} and \eqref{definitionnk}
together with the behavior of the mapping functions $\xi_0, \xi_1, \xi_2$ around the branch
points we find \eqref{behaviorN} for the case $w_1 \equiv 1$, $w_2 \equiv 1$,
(in which case of course $\alpha = \beta = \gamma = 0$).

For the general case we first note that the functions
$D_j(z; h_1, h_2)$ remain bounded and bounded away from $0$ for analytic
and non-zero $h_1$ and $h_2$.
For the Szeg\H{o} functions \eqref{DfunctionsABC} associated with the pure
Jacobi weights we have
\begin{equation}
    \begin{aligned}
    \begin{pmatrix} D_0(z; w_1^{(\alpha,\beta)}, w_2^{(\beta, \gamma)}) &
     D_1(z; w_1^{(\alpha,\beta)}, w_2^{(\beta, \gamma)}) &
      D_2(z; w_1^{(\alpha,\beta)}, w_2^{(\beta, \gamma)}) \end{pmatrix} \\
    = \begin{cases} \mathcal{O} \begin{pmatrix} (z-a)^{-\alpha/2} & (z-a)^{\alpha/2} & 1 \end{pmatrix}
        & \textrm{ as } z \to a,  \\
      \mathcal{O} \begin{pmatrix} (z-1)^{-\gamma/2} & 1 & (z-1)^{\gamma/2} \end{pmatrix}
        & \textrm{ as } z \to 1,  \\
      \mathcal{O} \begin{pmatrix} z^{-2\beta/3} & z^{\beta/3} & z^{\beta/3} \end{pmatrix}
        & \textrm{ as } z \to 0,
        \end{cases}
        \end{aligned}
        \end{equation}
where we use that
\[ \xi_j(z) = \mathcal O(z^{1/3})  \textrm{ as } z \to 0, \qquad \text{for } j=0,1,2, \]
\[ \xi_j(z) = \mathcal O((z-a)^{1/2}) \textrm{ as } z\to a, \qquad \text{for } j=0,1, \]
\[ \xi_j(z) = \mathcal O((z-1)^{1/2}) \textrm{ as } z\to 1, \qquad \text{for } j=0,2, \]
while $\xi_1(z)$ is analytic around $z=1$ and $\xi_2(z)$ is analytic around $z=a$.
\end{proof}

\subsubsection{Symmetries in the outer parametrix}

In this subsection we give two symmetries in the functions $\widetilde{N}$
that will be useful later on.  Recall that $\widetilde{N}$ is the outer
parametrix in the case that $w_1$ and $w_2$ are identically one, see subsection \eqref{sectionSpecialN}.
A first symmetry deals with the inverse of $\widetilde{N}$.

\begin{proposition}
For every $a < 0$ we have
\begin{equation} \label{Ntildesymmetry}
    \widetilde{N}^{-1}(z) = \widetilde{N}^{\transpose}(z),
    \qquad z \in \mathbb C \setminus [a,1]. \end{equation}
\end{proposition}
\begin{proof}
Define  $X$ by
\begin{equation}
X(z) := \widetilde{N}(z) \widetilde{N}^{\transpose}(z) \quad z \in \mathbb{C} \setminus [a,1]
\end{equation}
Using the fact $J_{\widetilde{N}} = \left(J_{\widetilde{N}}\right)^{-\transpose}$ we find that on $(a,0)$ and $(0,1)$,
\begin{align}
X_{-}^{-1} X_{+} & = \widetilde{N}_{-}^{-\transpose} \widetilde{N}_{-}^{-1} \widetilde{N}_{+} \widetilde{N}_{+}^{\transpose}
      = \widetilde{N}_{-}^{-\transpose} J_{\widetilde{N}} \widetilde{N}_{+}^{\transpose} \nonumber \\
& = \widetilde{N}_{-}^{-\transpose} \left( J_{\widetilde{N}} \right)^{-\transpose} \widetilde{N}_{+}^{\transpose}
    = \left( \widetilde{N}_{+} J_{\widetilde{N}}^{-1} \widetilde{N}_{-}^{-1} \right)^{\transpose} = I^{\transpose} = I
\end{align}
Since $\widetilde{N}(z)$ tends to $I$ as $z \to \infty$ we have
$X(z) = I + \mathcal{O}\left( \frac{1}{z} \right)$ as $z \to \infty$.

By the behavior of $\widetilde{N}$ near the branch points  (see \eqref{behaviorN}
for the case $\alpha = \beta = \gamma =0$) we obtain that $X$ has no poles in $a$, $0$ or $1$,
and we conclude by  Liouville's theorem that $X(z) = I$ everywhere and \eqref{Ntildesymmetry}
follows.
\end{proof}

As a corollary of this proposition and the expression for the $D_j$ in \eqref{expressiond} we
then  also find the behavior of $N^{-1}$ around the branch points, since from \eqref{ddinfty}
and \eqref{Ntildesymmetry}
\begin{equation}
    N^{-1}(z) = D(z)^{-1} \widetilde{N}(z)^{\transpose} D_{\infty}.
\end{equation}

We then obtain
\begin{equation}   \label{behaviorNinv}
  N^{-1}(z) =\begin{cases}
        \mathcal{O}\begin{pmatrix}
        (z-a)^{\frac{2\alpha-1}{4}} & (z-a)^{\frac{2\alpha -1}{4}} & (z-a)^{\frac{2\alpha -1}{4}}  \\
        (z-a)^{-\frac{2\alpha +1}{4}} & (z-a)^{-\frac{2\alpha +1}{4}} & (z-a)^{-\frac{2\alpha +1}{4}} \\
        1 & 1 & 1 \end{pmatrix} & \textrm{ as } z \to a , \\
        \mathcal{O}\begin{pmatrix}
    (z-1)^{\frac{2\gamma -1}{4}} & (z-1)^{\frac{2\gamma -1}{4}} & (z-1)^{\frac{2\gamma -1}{4}}  \\
    1 & 1 & 1 \\
    (z-1)^{-\frac{2\gamma +1}{4}} & (z-1)^{-\frac{2\gamma +1}{4}} & (z-1)^{-\frac{2\gamma +1}{4}}
    \end{pmatrix} & \textrm{ as } z \to 1, \\
    \mathcal{O}\begin{pmatrix}
    z^{\frac{2\beta -1}{3}} & z^{\frac{2\beta -1}{3}} & z^{\frac{2\beta -1}{3}} \\
    z^{-\frac{\beta+1}{3}} & z^{-\frac{\beta+1}{3}} & z^{-\frac{\beta+1}{3}} \\
    z^{-\frac{\beta+1}{3}} & z^{-\frac{\beta+1}{3}} & z^{-\frac{\beta+1}{3}}
    \end{pmatrix} & \textrm{ as } z \to 0.
    \end{cases} \end{equation}

A second symmetry relates the functions $\widetilde{N}$ for different values of $a$ to each
other. We use $\widetilde{N}(\cdot;a)$ to denote the dependence on $a < 0$.
 Let $\theta_{a}$ be the M\"{o}bius transformation leaving $0$ and $1$ invariant and mapping $-1$ to $a$, i.e.,
\begin{equation}
    \theta_{a}(z) = \frac{2az}{(a+1)z + a-1}, \qquad \theta_{a}^{-1}(z) = \frac{ (1-a) z }{(a+1)z - 2a}
\end{equation}

\begin{proposition} \label{propna}
For $a < 0, a \neq -1$ and any $z \in \mathbb{C} \setminus [a,1]$ we have
\begin{equation} \label{hom1}
\widetilde{N}(z;a) = \widetilde{N}\left( \frac{1-a}{1+a}; -1 \right)^{-1} \widetilde{N}\left(\theta_{a}^{-1}(z);-1\right)
\end{equation}
\end{proposition}

\begin{proof}
The proof is similar to the proof of the previous proposition. We do not give details.
\end{proof}

\subsubsection{Behavior at $z=0$}
In the next section we need the leading term for $\widetilde{N}(z;a)$ as $z \to 0$.
\begin{lemma}
As $z \to 0$ we have
\begin{equation}
    \widetilde{N}(z;a) = \widetilde{N}_0(a) z^{-\frac{1}{3}} + \mathcal{O}(1)
            \label{leadingtermsn1} \end{equation}
with
\begin{multline}
    \widetilde{N}_0(a) = \frac{1}{3 2^{1/6}}
        \widetilde{N} \left( \frac{1-a}{1+a};-1\right)^{-1} \begin{pmatrix} \sqrt{2}i \\ 1 \\ -1 \end{pmatrix}  \\
        \times
    \begin{cases}
         \begin{pmatrix} -\omega & 1 & \omega^{2} \end{pmatrix} &  \textrm{ for } \im z > 0, \\
       \begin{pmatrix} \omega^{2} & 1 & \omega \end{pmatrix} & \textrm{ for } \im z < 0.
       \end{cases} \label{leadingtermsn2}
\end{multline}
\label{lemmaleadingtermsn}
\end{lemma}

\begin{proof}
By Proposition \ref{propna} and the fact that $\theta_{a}^{-1}(z) = z + \mathcal{O}(z^{2})$
it suffices to compute the
leading term of $\widetilde{N}(z;-1)$ as $z \to 0$.  We use the factorization (where we suppress the argument $z$ on the right hand side)
\begin{multline}
    \widetilde{N}(z;-1)  = \diag \begin{pmatrix} 1 & \frac{p_1(p_1^2-1)^{1/2}}{p_1-p_2} & \frac{p_2 (p_2^2-1)^{1/2}}{p_2-p_1} \end{pmatrix} \\
    \times
     \begin{pmatrix} (\xi_{0}- p_{1})(\xi_{0}- p_{2}) & (\xi_{1}- p_{1})(\xi_{1}- p_{2}) & (\xi_{2}- p_{1})(\xi_{2}- p_{2}) \\
     \xi_{0} - p_{2} & \xi_{1} - p_{2} & \xi_{2} - p_{2} \\ \xi_{0} - p_{1} & \xi_{1} - p_{1} & \xi_{2} - p_{1} \end{pmatrix}
     \\
   \times \diag \begin{pmatrix} \frac{1}{\xi_0 (\xi^2_{0}-1)^{\frac{1}{2}}} & \frac{1}{\xi_1 (\xi^2_1 -1)^{\frac{1}{2}}} &
    \frac{1}{\xi_2 (\xi_{2}^2-1)^{\frac{1}{2}}} \end{pmatrix},
\end{multline}
A careful analysis of all the functions and constants involved then shows that for $\im z > 0$
\begin{equation}
    \widetilde{N}(z;-1) = \frac{1}{3 2^{1/6}} z^{-\frac{1}{3}} \begin{pmatrix} \sqrt{2}i \\ 1 \\ -1 \end{pmatrix}
    \begin{pmatrix} -\omega & 1 & \omega^{2} \end{pmatrix}  + \mathcal{O}(1) \textrm{ as } z \to 0, \end{equation}
    and for $\im z < 0$
\begin{equation}
    \widetilde{N}(z;-1) = \frac{1}{3 2^{1/6}} z^{-\frac{1}{3}} \begin{pmatrix} \sqrt{2}i \\ 1 \\ -1 \end{pmatrix}
    \begin{pmatrix} \omega^{2} & 1 & \omega \end{pmatrix}  + \mathcal{O}(1) \textrm{ as } z \to 0. \end{equation}
Together with \eqref{hom1} this proves the lemma.
\end{proof}

\subsection{Local parametrices}
\label{sectionlocal}

\subsubsection{Local parametrices around $\pm 1$}
The outer parametrix $N$ is intended as an approximation to $S$.
However the approximation cannot be good around the branch points $a$, $0$ and $1$.
Indeed, the entries of $S(z)N^{-1}(z)$ will typically diverge as $z$ tends to one of the branch points.

The solution to this problem is building local approximations around the branch points,
called local parametrices.  The appropriate construction around $a$ and $1$ is standard, and
uses the Bessel model parametrix as defined in \cite{kmvv04}, equations (6.23)-(6.25).
Let $U_{-1}$ and $U_{1}$ be disks around respectively $-1$ and $1$ of fixed but small enough radius:
$U_{1}$ should be contained in $V_{2}$, such that $w_{2}$ is well-defined on
$U_{1} \setminus [1, + \infty)$.  Similarly we must have $U_{-1} \subset V_{1}$, and
additionally $a$ must lie inside $U_{-1}$.  On these disks we construct $3 \times 3$ matrix valued functions
$P_{-1}$ and $P_{1}$ that satisfy the same jumps as $S$, see \eqref{jS1}--\eqref{jS2},
and match with $N$ on the boundary of the disks:
\begin{equation}
    \begin{aligned}
    P_{-1}(z) N(z)^{-1} & = I + \mathcal{O}\left(\frac{1}{n}\right) \quad \textrm{ for } z \in \partial U_{-1}, \\
    P_{1}(z) N(z)^{-1} & = I + \mathcal{O}\left(\frac{1}{n}\right) \quad \textrm{ for } z \in \partial U_{1},
    \end{aligned}  \label{matchinga1}
    \end{equation}
    as $n \to \infty$. The $\mathcal{O}$-terms are uniform in $z$.  For details of the construction of $P_{-1}$ and $P_{1}$
we refer to \cite{kmvv04}, where the Bessel model parametrix was introduced, and \cite{lywi08}, where it
was also used in a $3 \times 3$ matrix valued Riemann-Hilbert problem.

\subsubsection{Local parametrix around $0$: statement}

Around $0$ we need a new kind of local parametrix $P_{0}$ on a disk $U_0$ around the origin.
There are a number of difficulties to obtain the desired matching condition
\[ P_0(z) N(z)^{-1} = I + \mathcal O(n^{-\kappa}) \quad \text{ for } z \in \partial U_0 \]
with some $\kappa > 0$, that in fact we are unable to resolve.
The best we can do is to construct $P_0$ such that $P_0(z) N(z)^{-1}$ remains bounded as $n \to \infty$
for $z$ on a circle of radius that decays like $n^{-1/2}$ as $n \to \infty$.

Thus the disk $U_0$ should be shrinking as $n$ increases, and for definiteness we take
\begin{equation}
    \textrm{ radius } U_{0} := n^{-\frac{1}{2}} , \end{equation}
and we assume $n$ to be large enough so that $U_{0}$ is contained in $V_{1} \cap V_{2}$.
Then consider the following Riemann-Hilbert problem for the local parametrix $P_0$ around $0$.
\begin{itemize}
\item $P_0$ is analytic on $U_{0} \setminus \Sigma_S$,
\item $P_0$ has  jumps
\begin{equation} \label{jP}
        P_{0,+} = P_{0,-} J_P \qquad \text{on } \Sigma_S \cap U_0, \qquad \text{where } J_P = J_S,
        \end{equation}
    see \eqref{jS1}-\eqref{jS2},
\item $P_0(z)$ behaves in the same way as $S(z)$ as $z \to 0$, see \eqref{Snear0},
\item $P_0 N^{-1}$ remains bounded on the boundary of $U_{0}$,
\begin{equation}  \label{matchingcondition}
    P_0(z) N(z)^{-1}  = \mathcal{O}(1) \quad \textrm{ for } z \in \partial U_{0},
    \end{equation}
    as $n \to \infty$, where $a = a_n$  depends on $n$ as in \eqref{an}.
\end{itemize}

Notice that the matching between $P_0(z)$ and $N(z)$ does not improve
with increasing $n$. Indeed, the matrix $P_0(z) N(z)^{-1}$ does
not tend to $I$ as $n \to \infty$ for $z \in \partial U_{0}$.
The matching \eqref{matchingcondition} is the best we can obtain
without modifying the outer parametrix $N$.  However, with $a=a_n$ as in \eqref{an}
we will be able to find a $3 \times 3$ matrix valued function $Z_{n}(\cdot; a)$ such that
\begin{equation}
    P_0(z) N(z)^{-1} = I + Z_n(z;a) + \mathcal{O}\left(n^{-1/6}\right) \quad \textrm{ for } z \in \partial U_{0} \label{appearancez}
\end{equation}
The explicit expression and special properties of $Z_n$ will allow us to create, in the final transformation,
a jump on $\partial U_{0}$ that tends to $I$ as $n \to \infty$.

Since the dependence on $a$ will be important, we emphasize that most notions
depend on $a$ and have limiting values as $a \to -1$. As before, we will not always
explicitly indicate the dependence on $a$, but sometimes we do.

\subsubsection{Reduction to constant jumps}

We factor out the $\varphi_{i}$ and $w_{i}$-functions from the jump matrices \eqref{jS1}-\eqref{jS2}.
Define for $z \in U_{0}$ the matrix valued functions
\begin{align}
    \Lambda(z) & := \frac{2}{3} \begin{pmatrix} \varphi_{1}(z) + \varphi_{2}(z) & 0 & 0 \\
    0 & \varphi_{2}(z) - 2 \varphi_{1}(z) & 0 \\
    0 & 0 & \varphi_{1}(z) - 2 \varphi_{2}(z) \end{pmatrix}, \label{definitionlambda} \\
     W(z) & :=
    \begin{pmatrix} z^{\beta} & 0 & 0 \\ 0 & (z-a)^{-\alpha} h_{1}(z)^{-1} & 0 \\ 0 & 0 & (1-z)^{-\gamma} h_{2}(z)^{-1} \end{pmatrix}.
    \label{definitionw}
    \end{align}
We look for $P_0$ in the form
\begin{equation} \label{Ptilde}
    P_0(z) =  \widetilde{P}_0(z) e^{n\Lambda(z)} W^{-1}(z).
\end{equation}
In order that $P_0$ has the jumps $J_P$, we should have
$\widetilde{P}_{0,+} = \widetilde{P}_{0,-} J_{\widetilde{P}}$ with
\begin{equation}  \label{jumpsm1}
    J_{\widetilde{P}} = \begin{cases}
        \begin{pmatrix} 1 & 0 & 0 \\ e^{\pm \beta \pi i} & 1 & 0 \\ 0 & 0 & 1 \end{pmatrix} \textrm{ on } \Sigma_{1}^{\pm} \cap U_{0}, \\
        \begin{pmatrix}1 & 0 & 0 \\ 0 & 1 & 0 \\ 1 & 0 & 1 \end{pmatrix} \textrm{ on } \Sigma_{2}^{\pm} \cap U_{0}, \\
        \begin{pmatrix} 0 & e^{\beta \pi i} & 0 \\ - e^{\beta \pi i} & 0 & 0 \\ 0 & 0 & 1 \end{pmatrix} \textrm{ on } [a,0] \cap U_{0}, \\
        \begin{pmatrix} 0 & 0 & 1 \\ 0 & 1 & 0 \\ - 1 & 0 & 0 \end{pmatrix} \textrm{ on } [0,1] \cap U_{0}.
        \end{cases}
 \end{equation}

\subsubsection{Functions $f(z)$ and $\tau(z)$}

Note that the jumps \eqref{jumpsm1} are exactly the same as the ones for $\Psi$, see Figure~\ref{figurejumpspsi},
except that the jumps for $\Psi$ are on unbounded rays. Recall that $\Psi(z; \tau)$ also depends
on $\tau$ which appears in the asymptotic condition \eqref{definitionTheta}.

Our aim is to construct $\widetilde{P}_0$ of the form
\begin{equation}  \label{Ptildeform}
    \widetilde{P}_0(z) = E_n(z) \Psi\left(n^{\frac{3}{2}} f(z); n^{\frac{1}{2}} \tau(z)\right),
\end{equation}
where $f(z)$ is a conformal map and $\tau(z)$ is analytic in $U_{0}$.
The matrix valued function $E_n(z)$ is an analytic prefactor, which will be defined in the next subsection.
We are going to choose $f(z)$ and $\tau(z)$ such that
\begin{equation} \Theta\left(n^{\frac{3}{2}}f(z);n^{\frac{1}{2}} \tau(z)\right) + n \Lambda(z) = 0
    \quad \textrm{ for } z \in U_0 \setminus \mathbb{R} \label{droppingexp}
    \end{equation}
where $\Theta$  is given by \eqref{definitionTheta} and $\Lambda$ is given by \eqref{definitionlambda}.
When this condition is satisfied there will be no exponential growth (as $n \to \infty$)
in $P_0(z)$ see \eqref{Ptilde}, and so there is a chance that we can match it with $N$.

Define functions $\lambda_{1}(z)$ and $\lambda_{2}(z)$ on $\mathbb{C} \setminus \mathbb{R}$ by
\begin{align}
    \lambda_{1}(z) & := \begin{cases}
        - z^{-\frac{1}{3}} \left( \varphi_{1}(z) + \omega^{2} \varphi_{2}(z) \right) & \textrm{ for } \im z > 0, \\
        - z^{-\frac{1}{3}} \left( \varphi_{1}(z) + \omega \varphi_{2}(z) \right) & \textrm{ for } \im z < 0,
        \end{cases} \label{lambdavarphi1} \\
    \lambda_{2}(z) & := \begin{cases}
        - z^{-\frac{2}{3}} \left( \varphi_{1}(z) + \omega \varphi_{2}(z) \right) & \textrm{ for } \im z > 0, \\
        - z^{-\frac{2}{3}} \left( \varphi_{1}(z) + \omega^{2} \varphi_{2}(z) \right) & \textrm{ for } \im z < 0,
        \end{cases} \label{lambdavarphi2}
        \end{align}
where $\varphi_1$ and $\varphi_2$ are given by \eqref{definitionvarphi}.
It can be checked that the functions $\lambda_{1}(z)$ and $\lambda_{2}(z)$ have no jumps on
$(a,0)$ or $(0,1)$.  Since $\varphi_{1}$ and $\varphi_{2}$ are bounded, $\lambda_{1}$ and $\lambda_{2}$
have analytic continuations to $\mathbb{C} \setminus ((-\infty, a] \cup [1, \infty))$.
Since the $\varphi$-functions depend on $a$, so do the $\lambda$-functions,
and we write $\lambda(z;a)$ to emphasize this fact.

It may be checked that $\lambda_{1}(z;a)$ and $\lambda_{2}(z;a)$ converge
uniformly in a neighborhood of zero as $a \to -1$, and also that
\begin{equation}
    \begin{aligned} \lambda_{1}(0;a) & =  \frac{3 \cdot 2^{2/3}}{4} (a+1) + \mathcal{O}(a+1)^{2}, \\
        \lambda_{2}(0;a) & = \frac{9 \cdot 2^{1/3}}{4} + \mathcal{O}(a+1)
        \end{aligned} \label{lambda20}\end{equation}
as $a \to -1$. Then $\re \lambda_{2}(z;a) > 0$ for $a$ close enough to $-1$ and $z \in U_0$,
and we can define the following analytic functions in a neighborhood of $0$.
\begin{definition}
For $a$ close enough to $-1$ and $z \in U_0$ we define
\begin{equation}
    \begin{aligned}
    f(z) = f(z;a) & := \frac{8}{27} z \lambda_{2}(z;a)^{\frac{3}{2}}, \\
   \tau(z) = \tau(z;a) & := \frac{\lambda_{1}(z;a)}{\lambda_{2}(z;a)^{\frac{1}{2}}}.
   \end{aligned} \label{definitionftau} \end{equation}
\end{definition}

Then $f$ is a conformal map with $f(0) = 0$, and $f(z)$ is real for real arguments $z$. By \eqref{lambda20}
and \eqref{definitionftau} we have as $a \to -1$.
\begin{equation}  \label{ftauat0}
    \begin{aligned}
    f'(0;a) & = \sqrt{2} + \mathcal{O}(a+1), \\
    \tau(0;a) & = \frac{1}{\sqrt{2}}(a+1) + \mathcal{O}(a+1)^{2}.
    \end{aligned}
\end{equation}

Without loss of generality we can now assume that the lips of the lenses
are chosen such that $f$ maps $\Sigma_{S} \cap U_{0}$ into $\Sigma_{\Psi}$.
Then $\Psi\left(n^{\frac{3}{2}}f(z); n^{\frac{1}{2}} \tau(z) \right)$ is well-defined
and analytic in $U_{0} \setminus \Sigma_{S}$.  It remains to check the condition \eqref{droppingexp}.  By \eqref{definitionthetak} and \eqref{definitionftau} we find:
\begin{align}
\theta_{k} \left( n^{\frac{3}{2}} f(z); n^{\frac{1}{2}} \tau(z) \right) & = -\frac{3n}{2} \omega^{k} f(z)^{\frac{2}{3}} - n \tau(z) \omega^{2k} f(z)^{\frac{1}{3}} \notag\\
& = -\frac{2n}{3} \left(\omega^{k} z^{\frac{2}{3}} \lambda_{2}(z) + \omega^{2k} z^{\frac{1}{3}} \lambda_{1}(z) \right)\end{align}

By \eqref{lambdavarphi1} and \eqref{lambdavarphi2} the right hand sides are exactly minus $n$ times the components of $\Lambda$, and \eqref{droppingexp} follows.

\subsubsection{Prefactor $E_n(z)$}

Next we define a suitable analytic prefactor $E_n(z) = E_n(z;a)$ such that the local parametrix
\begin{equation}
    P_0(z) = E_n(z) \Psi\left(n^{\frac{3}{2}} f(z); n^{\frac{1}{2}} \tau(z)\right) e^{n\Lambda(z)} W(z)^{-1}
    \label{definitionp} \end{equation}
     satisfies the matching condition \eqref{matchingcondition} with $N$ on $\partial U_{0}$.
     Also we obtain an expression for the function $Z_n$ in \eqref{appearancez}.

Denote by $A(z;\tau)$ the right-hand side of \eqref{asymptoticexpansionpsi}
without the exponential factor $e^{\Theta(z; \tau)}$ and the error factor $I + \mathcal{O}(z^{-\frac{1}{3}})$.
Thus
\begin{equation}
    A(z; \tau) := \sqrt{\frac{2 \pi}{3}} e^{\tau^{2}/6} z^{ \frac{\beta}{3}}
    \begin{pmatrix} z^{\frac{1}{3}} & 0 & 0 \\ 0 & 1 & 0 \\ 0 & 0 & z^{-\frac{1}{3}} \end{pmatrix}
    \Omega_{\pm} B_{\pm}, \qquad \pm \im z > 0. \label{definitiona}
\end{equation}

Then we define
\begin{definition}
For $a$ close enough to $-1$ and $z \in U_0$ we define
\begin{equation}
    E_n(z;a) = N(z;a) W(z;a) A^{-1}\left(n^{\frac{3}{2}} f(z;a); n^{\frac{1}{2}} \tau(z;a) \right).
    \label{En}
\end{equation}
where $f(z;a)$ and $\tau(z;a)$ are given by \eqref{definitionftau}.
\end{definition}

\begin{proposition}
The function $E_n$ defined by \eqref{En} is analytic on $U_{0}$.\label{propositioneisanalytic} \end{proposition}
\begin{proof}

The function $A$ from \eqref{definitiona} has jumps on the real line, given by $A_{+} = A_{-} J_{A}$ with
\begin{equation}
\begin{aligned}
    J_{A} = \begin{pmatrix} 0 & e^{\beta \pi i} & 0 \\  - e^{\beta \pi i} & 0 & 0 \\ 0 & 0 & 1 \end{pmatrix}
        \textrm{ on } \mathbb R^-, \quad
        J_{A}= \begin{pmatrix} 0 & 0 & 1 \\ 0 & 1 & 0 \\ - 1 & 0 & 0 \end{pmatrix}
    \textrm{ on } \mathbb R^+. \end{aligned}
    \end{equation}
Then $A\left(n^{\frac{3}{2}} f(z); n^{\frac{1}{2}} \tau(z) \right)$ has the
corresponding jumps on $\mathbb{R} \cap U_0$.  A straightforward calculation reveals that $N(z) W(z)$ has
exactly the same jumps on $U_{0} \cap (a,0)$ and $U_{0} \cap (0,1)$, and hence $E_n(z)$ has no
branch cuts in $U_{0}$.

Also $E_n$ can have no pole in $0$.  By \eqref{behaviorN} and the definition of $W$ \eqref{definitionw} we have that
\begin{equation} N(z) W(z) = \mathcal{O}\left( z^{\frac{\beta -1}{3}}\right) \textrm{ as } z \to 0 \label{beh1} \end{equation}
From \eqref{definitiona} and the fact that $f$ is a conformal map with $f(0) = 0$ we obtain
\begin{equation} A\left( n^{\frac{3}{2}} f(z); n^{\frac{1}{2}} \tau(z) \right)^{-1} = \mathcal{O} \left( z^{-\frac{\beta + 1 }{3}}\right) \label{beh2} \end{equation}
Then by \eqref{En} and  \eqref{beh1} and \eqref{beh2} we see that $E_n(z) = \mathcal{O}(z^{-\frac{2}{3}})$, and
therefore the isolated singularity at $0$ is removable.
\end{proof}

We remark that by a similar argument we have that $E_n^{-1}(z;a)$ is analytic in $U_0$
and in particular at $z=0$ as well.


\subsubsection{Matching condition}
\label{subsectionmatchingcondition}

We show that the matching condition \eqref{matchingcondition} holds, and
we compute $Z_n$ from \eqref{appearancez}.
\begin{proposition} \label{prop:Znza}
The parametrix $P_0$ defined by \eqref{Ptilde}, \eqref{Ptildeform}, \eqref{En}
satisfies the matching condition \eqref{matchingcondition} as $n \to \infty$
with $a= a_n$  as in \eqref{an}. The matching  \eqref{appearancez} holds with
\begin{multline}
    Z_n(z;a) = -\frac{1}{81 \cdot 2^{1/3}} \frac{ \tau(z;a) ( n \tau^{2}(z;a) + 9\beta)}{z^{\frac{2}{3}} f(z;a)^{\frac{1}{3}}} \\
    \times
    D_{\infty}^{-1} \begin{pmatrix} \sqrt{2}i \\ 1 \\ -1 \end{pmatrix} \begin{pmatrix} \sqrt{2}i& 1 & -1 \end{pmatrix}
    D_{\infty}. \label{definitionz}
\end{multline}
\end{proposition}
\begin{proof}
The local parametrix $P_0$, written in full, is given by
\begin{equation} \label{Pinfull}
    P_0(z) = N(z) W(z) A\left( n^{\frac{3}{2}} f(z); n^{\frac{1}{2}} \tau(z) \right)^{-1}
    \Psi\left(n^{\frac{3}{2}} f(z), n^{\frac{1}{2}} \tau(z) \right) e^{n\Lambda(z)} W(z)^{-1},
\end{equation}
where all functions also depend on $a$.

We first observe that $n^{\frac{1}{2}} \tau(z;a)$ remains bounded for $z \in U_0$ as $n \to \infty$.
Here we need the fact that we took the radius of $U_{0}$ to be $n^{-\frac{1}{2}}$, and
that $a = a_n = -1 + \mathcal{O}(n^{-1/2})$. Since $\tau(0;-1) = 0$, see \eqref{ftauat0}
we indeed obtain
\begin{equation}
    \tau(z;a) = \tau(z;-1) + \mathcal{O}\left(a-a_n\right) = \mathcal{O}(n^{-1/2}) \quad \text{for } |z| = n^{-1/2},
\end{equation}
as $n \to \infty$.

It follows that we can use the asymptotic expansion \eqref{refinedexpansion} for $\Psi(z;\tau)$
in \eqref{Pinfull}, since \eqref{refinedexpansion} is uniformly valid for $\tau$-values in a bounded set.
Then by combining \eqref{refinedexpansion} and the definition \eqref{definitiona}  of $A$ we get that
a number of factors cancel, and what remains from \eqref{Pinfull} is
\begin{multline}
 P_0(z) N(z)^{-1} = N(z) W(z) B_{\pm}^{-1} \\
    \times  \left( I + \frac{(\Psi_{1})_{\pm} (n^{\frac{1}{2}}\tau(z))}{n^{\frac{1}{2}} f(z)^{\frac{1}{3}}}
    + \mathcal{O} ( n^{-\frac{2}{3}}) \right) B_{\pm} W(z)^{-1} N(z)^{-1}\label{end1}
\end{multline}
as $n \to \infty$ uniformly for $z \in \partial U_0$.

For $z \in \partial U_{0}$ we have $|z| = n^{-\frac{1}{2}}$ and the entries
of $N W$ are $\mathcal{O} (z^{\frac{\beta -1}{3}})$ as $z \to 0$.
Hence $(NW)(z)$ is $\mathcal{O}(n^{\frac{1-\beta}{6}})$ for $z$ on $\partial U_{0}$ as $n\to \infty$.
For the inverse of $NW$ we get $\mathcal{O}(n^{- \frac{\beta +1}{6}})$ entries on $\partial U_{0}$.
Hence the $\mathcal{O}(n^{-\frac{2}{3}})$ term in the middle factor of the right hand side of \eqref{end1}
turns into a $\mathcal{O}(n^{-\frac{1}{3}})$ term:
\begin{multline}
 P_0(z) N(z)^{-1} = I + \frac{1}{n^{\frac{1}{2}} f(z)^{\frac{1}{3}}} N(z) W(z) B_{\pm}^{-1} \\
       \times (\Psi_{1})_{\pm}(n^{\frac{1}{2}} \tau(z)) B_{\pm} W(z)^{-1} N(z)^{-1}
        + \mathcal{O}\left( n^{-\frac{1}{3}} \right).\label{end2}
\end{multline}
In the same way we find that the second term in the right-hand side of \eqref{end2} is $\mathcal{O}(1)$
as $n \to \infty$. We evaluate this term in more detail.

By \eqref{ddinfty} we have
\[ N(z) W(z) B_{\pm}^{-1} = D_{\infty}^{-1} \widetilde{N}(z) D(z) W(z) B_{\pm}^{-1} \]
where by \eqref{factorizationn}--\eqref{ddinfty}, \eqref{definitionw}, and \eqref{definitionomega},
the last three matrices on the right-hand side are diagonal and their product satisfies for
some constant $c \neq 0$,
\[ D(z) W(z) B_{\pm}^{-1} = c z^{\beta/3} (I + \mathcal{O}(z)) \qquad \textrm{ as } z \to 0. \]
Furthermore the leading behavior of $\widetilde{N}(z)$ and its inverse as $z \to 0$
follow from \eqref{Ntildesymmetry} and \eqref{leadingtermsn1}.
Thus \eqref{end2} reduces to (where we emphasize again the dependence on $a$),
\begin{multline}
 P_0(z) N(z)^{-1} = I +  \frac{1}{n^{\frac{1}{2}} z^{\frac{2}{3}}f(z;a)^{\frac{1}{3}}} \\
 \times D_{\infty}^{-1}
            \widetilde{N}_0(a) (\Psi_{1})_{\pm}(n^{\frac{1}{2}} \tau(z;a)) \widetilde{N}_0^{T}(a) D_{\infty}
        + \mathcal{O}\left( n^{-\frac{1}{6}} \right).  \label{end3}
\end{multline}

By \eqref{leadingtermsn2} and \eqref{asymptoticconditionn} we have
\begin{align*} \widetilde{N}_0(a) = \frac{1}{3 \cdot 2^{1/6}} \left( I + \mathcal{O}(a+1) \right)
    \begin{pmatrix} \sqrt{2} i \\ 1 \\ -1 \end{pmatrix} \times
         \begin{cases}
            \begin{pmatrix} -\omega & 1 & \omega^2 \end{pmatrix} & \textrm{for } \im z  > 0, \\
      \begin{pmatrix} \omega^2 & 1 & \omega \end{pmatrix}  & \textrm{for } \im z < 0,
      \end{cases} \end{align*}
as $a \to -1$. We plug this and the explicit formulas for $(\Psi_1)_{\pm}$ (see Lemma \ref{lem:Psi1})
into \eqref{end3}. Then after some calculations we indeed obtain
\eqref{appearancez} with $Z_n(z;a)$ given by \eqref{definitionz}.

Since $n \tau^2(z;a)$ remains bounded as $n \to \infty$ if $|z| = n^{-1/2}$ and $a = a_n = -1 + \mathcal{O}(n^{-1/2})$,
we obtain that $Z_n(z;a) = \mathcal{O}(1)$. This proves the proposition.
\end{proof}

We note from \eqref{definitionz} that $Z_n(z;a)$ is analytic in a punctured neighborhood of  $z=0$
with a simple pole at $z=0$.  
It also follows from \eqref{definitionz} that
\begin{equation} \label{eq:Znsquare} 
	Z_n(z_1;a) Z_n(z_2;a) = 0 \qquad \text{for all } z_1, z_2 \text{ near } 0. 
	\end{equation}
This property will be important in the final transformations

\section{Final transformations}
\label{sectionfinaltransfo}

We will do the final transformation $S \mapsto R$ in two steps.
First we will define $R_{0}$ as the approximation error between $S$ and the parametrices
$N, P_{-1}, P_{0}$ and $P_{1}$.
The jump matrices of $R_{0}$ will tend to the identity matrix as $n \to \infty$ on all parts of
the jump contour $\Sigma_R$, except on $\partial U_{0}$.

Via a global transformation we finally define $R$, in such a way that it also has a jump
that tends to the identity matrix on $U_0$. As a result of the steepest descent analysis,
we then derive a global and uniform estimate for this function $R$.

\subsection{Transformation $S \mapsto R_0$}

Define $R_{0}(z)$ as
\begin{equation} \label{defR0}
    R_{0}(z) = \begin{cases}
        S(z) N^{-1}(z), & z \in \mathbb{C} \setminus (\Sigma_{S} \cup \overline{U_{-1}} \cup \overline{U_{0}} \cup \overline{U_{1}}), \\
        S(z) P_{-1}^{-1}(z), & z \in U_{-1} \setminus \Sigma_{S}, \\
        S(z) P_{0}^{-1}(z), & z \in U_{0} \setminus \Sigma_{S}, \\
        S(z) P_{1}^{-1}(z), & z \in U_{1} \setminus \Sigma_{S}. \end{cases}
         \end{equation}

Then $R_{0}$ is defined and analytic on $\mathbb{C}$ minus the interval $[a,1]$, the lenses $\Sigma_{1,2}^{\pm}$,
and the circles $\partial U_{-1}$, $\partial U_{0}$, and $\partial U_{1}$.  By comparing the jumps of $S$
and the parametrices we find that $R$ has analytic continuation
into each of the disks, and across the parts of the real intervals $(a,0)$ and $(0,1)$ outside
of the disks. The singularities at $a$, $0$ and $1$ are removable. For $a$ and $1$ this
follows from the behavior of the Bessel parametrix given in \cite{kmvv04}. For
$0$ it requires a special check involving the behavior of $\Psi$ and $\Psi^{-1}$
that we will not give here.   The function $R_{0}$ will then have jumps along the reduced contour $\Sigma_{R}$
shown in Figure \ref{figurer}.

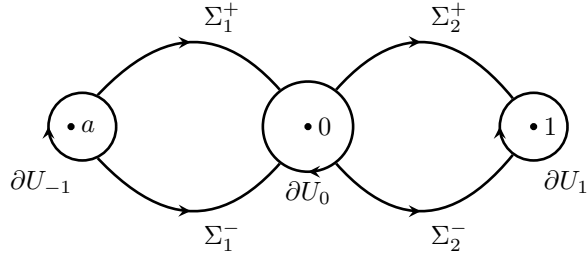
\begin{figure}[t]
\centering
\begin{tikzpicture}[scale = 3,line width = 1, decoration = {markings, mark = at position .5 with {\arrow{stealth};}} ]
\draw [postaction = {decorate}](0,0) .. controls (.3333,.5) and (.6666,.5) .. node[above right]{$\Sigma_{2}^{+}$}(1,0);
\draw [postaction = {decorate}](0,0) .. controls (.3333,-.5) and (.6666,-.5) .. node[below right]{$\Sigma_{2}^{-}$}(1,0);
\draw [postaction = {decorate}](-1.05,0) .. controls (-.6666,.5) and (-.3333,.5) .. node[above right]{$\Sigma_{1}^{+}$}(0,0);
\draw [postaction = {decorate}](-1.05,0) .. controls (-.6666,-.5) and (-.3333,-.5) .. node[below right]{$\Sigma_{1}^{-}$} (0,0);
\filldraw[white] (0,0) circle (.20);
\draw (0,-.2) node[below] {$\partial U_{0}$};
\filldraw[white] (-1,0) circle (.15);
\draw (-1,-.15) node[below left] {$\partial U_{-1}$};
\draw [postaction = {decorate}] (-.85,0) arc (0:-360:.15);
\draw [postaction = {decorate}] (0,.2) arc (90:-270:.2);
\filldraw[white] (1,0) circle (.15);
\draw (1,-.15) node[below right] {$\partial U_{1}$};
\draw [postaction = {decorate}] (1.15,0) arc (0:-360:.15);
\filldraw[black] (-1.05,0) circle (0.010) ;
\draw (-1.05,0) node[right]{$a$};
\filldraw[black] (0,0) circle (0.010);
\draw (0,0) node[right]{$0$};
\filldraw[black] (1,0) circle (0.010);
\draw (1,0) node[right]{$1$};
\end{tikzpicture}
\caption{The contour $\Sigma_{R}$ that consists of the circles $\partial U_{-1}$, $\partial U_0$, $U_1$
and the parts of the lenses $\Sigma_1^{\pm}$ and $\Sigma_2^{\pm}$ outside the disks.}
\label{figurer}
\end{figure}

We choose clockwise orientation for the circles. The lips of the lenses are oriented from left to
right, as before. Then $R_0$ satisfies the following RH problem.

\begin{itemize}
\item $R_{0}$ is defined and analytic on $\mathbb{C}\setminus \Sigma_{R}$,
\item $R_{0}$ satisfies the jump relation $R_{0,+} = R_{0,-} J_{R_{0}}$
    on $\Sigma_{R}$ with
\begin{equation}
    J_{R_{0}}(z) = \begin{cases} N(z) J_{S}(z) N^{-1}(z), &
            z \in \Sigma_R \setminus (\partial {U}_{-1} \cup \partial {U}_{0} \cup \partial {U}_{1}), \\
        P_{-1}(z) N^{-1}(z), & z \in  \partial U_{-1}, \\
        P_{0}(z) N^{-1}(z), & z \in \partial U_{0}, \\
        P_{1}(z) N^{-1}(z), & z \in \partial U_1,
        \end{cases} \label{jumpsr0} \end{equation}
\item $R_{0}(z) = I + \mathcal{O}\left(z^{-1}\right)$ as $z \to \infty$.
\end{itemize}

Due to the matching conditions for the local parametrices \eqref{matchinga1} and \eqref{matchingcondition} we have
\begin{equation}
    J_{R_{0}}(z) = I + \mathcal{O} \left( \frac{1}{n} \right) \textrm{ uniformly for } z \in \partial U_{-1} \cup \partial U_{1} \label{gor1} \end{equation}

On the lips of the lenses the off-diagonal entries of the jumps of $S$ involve
the functions $\varphi_{1,2}$, and they do not necessarily have
negative real parts along the lenses $\Sigma^{\pm}_{1,2}$.
However, due to the estimates in Lemma \ref{lem:revarphij} we may conclude that
for $a = a_n = -1 + \mathcal{O}(n^{-1/2})$, we have for some constant $C > 0$,
\[ \re \varphi_j(z;a_n) <  - C n^{-1/3},  \qquad \text{for } z \in \Sigma_R \setminus
    \left(\partial U_{-1} \cup \partial U_0 \cup \partial U_1\right) \]
Then by \eqref{jS2} we have $J_{S} = I + \mathcal{O}(e^{- C n^{2/3}})$, and
so by \eqref{jumpsr0} we also obtain
\begin{equation} \label{gor2}
        J_{R_{0}}(z) = I + \mathcal{O} \left( e^{-C n^{2/3}} \right) \text{ as } n \to \infty
\end{equation}
with a possibly different constant $C > 0$. The $\mathcal{O}$-term in  \eqref{gor2} holds uniformly
for $z \in \Sigma_R \setminus (\partial U_{-1} \cup \partial U_0 \cup \partial U_{1})$.

Due to \eqref{appearancez} and \eqref{jumpsr0} the jump matrix $J_{R_0}$ on $\partial U_{0}$
is given by
\begin{equation}
        J_{R_{0}}(z) = P_{0}(z) N(z)^{-1} = I + Z_n(z;a) + \mathcal{O}\left( n^{-\frac{1}{6}}\right)
        \textrm{ for } z \in \partial U_0,
\end{equation}
with $Z_n(z;a)$ given by \eqref{definitionz}, and the $\mathcal{O}$-term is valid under
the assumption that $a=a_n = -1 + \mathcal{O}(n^{-1/2})$ as $n \to \infty$, see Proposition \ref{prop:Znza}
The jump matrix on $\partial U_0$ does not converge to the identity matrix as $n \to \infty$.
Therefore we need one more transformation.

\subsection{Transformation $R_0 \mapsto R$}

Let $Z^{(0)}_n(a)$ denote the residue of $Z_n(z;a)$ at the simple pole $z =0$. Thus by
\eqref{definitionz}
\begin{equation} \label{definitionz0}
\begin{aligned}
     Z^{(0)}_n(a)  = \lim_{z\to 0} z Z_n(z;a)
    & = - \frac{1}{81 \cdot 2^{1/3}} \frac{ \tau(0;a) (n \tau^2(0;a) + 9 \beta)}{f'(0;a)^{1/3}}
        \\ & \qquad         \times
    D_{\infty}^{-1} \begin{pmatrix} \sqrt{2}i \\ 1 \\ -1 \end{pmatrix} \begin{pmatrix} \sqrt{2}i& 1 & -1 \end{pmatrix}
    D_{\infty}.
\end{aligned}
\end{equation}
We define the new matrix-valued function $R$ as
\begin{equation} \label{defRz}
    R(z) = R_0(z) \times \begin{cases} \ds  I - \frac{Z^{(0)}_n(a)}{z}, &
    \textrm{ for } z \in \mathbb C \setminus (\Sigma_{R} \cup \overline{U}_{0}), \\
    \ds I +  Z_n(z;a) - \frac{Z^{(0)}_n(a)}{z}, & \textrm{ for } z \in U_{0}.
    \end{cases}
\end{equation}
Then $R$ is defined and analytic in $\mathbb{C} \setminus \Sigma_{R}$,
and satisfies
\begin{itemize}
\item $R$ is analytic on $\mathbb{C}\setminus \Sigma_{R}$,
\item $R$ satisfies the jump relation $R_{+} = R_{-} J_{R}$
    on $\Sigma_{R}$ with
   \begin{equation} \label{jumpsr}
   J_R(z) = \begin{cases}
    \left(I + Z_n(z;a) - \frac{Z^{(0)}_n(a)}{z}\right)^{-1} J_{R_0}(z) \left(I + \frac{Z^{(0)}_n(a)}{z} \right)
        & z \in \partial U_0, \\
    \left(I - \frac{Z^{(0)}_n(a)}{z} \right)^{-1} J_{R_0}(z)
    \left(I - \frac{Z^{(0)}_n(a)}{z} \right) &
       z \in \Sigma_R \setminus \partial U_0.
       \end{cases} \end{equation}
\item $R(z) = I + \mathcal{O}\left(z^{-1}\right)$ as $z \to \infty$.
\end{itemize}

Note that by \eqref{definitionz} and \eqref{definitionz0} both $Z_n(z;a)$
and $Z^{(0)}_n(a)$ are scalar multiples of the constant
matrix $D_{\infty}^{-1} \begin{pmatrix} \sqrt{2}i & 1 & -1 \end{pmatrix}^T \begin{pmatrix} \sqrt{2}i& 1 & -1 \end{pmatrix}
    D_{\infty}$ whose square is zero. Therefore we have
such relations as
\begin{equation} \label{IplusZinv}
    \left(I + Z_n(z;a) - \frac{Z^{(0)}_n(a)}{z}\right)^{-1} = I - Z_n(z;a) + \frac{Z^{(0)}_n(a)}{z}
    \end{equation}
and $Z_n(z;a)^2 = Z_n(z;a) Z^{(0)}_n(a) = Z^{(0)}_n(a) Z_n(z;a) = (Z^{(0)}_n(a))^2 = 0$.

We also recall that for $a = a_n$ as in \eqref{an} we have by \eqref{ftauat0}
that $n^2 \tau (0;a_n)$ remains bounded as $n \to \infty$, while $f'(0,a_n)$ tends to $\sqrt{2}$.
Therefore by \eqref{definitionz0}
\begin{equation} \label{Z0decay}
    Z^{(0)}_n(a_n) = \mathcal{O}( n^{-1/2}) \quad \textrm{ as } n \to \infty.
    \end{equation}
This implies by \eqref{jumpsr0} that the same estimates as we had in \eqref{gor1}, \eqref{gor2} for
$J_{R_0}$ remain valid for $J_R$ on $\Sigma_R \setminus \partial U_0$. That is,
\begin{align} \label{jR1}
        J_{R}(z)  = \begin{cases}  I+
        \mathcal{O}( n^{-1})  & \textrm{ for } z \in \partial U_{-1} \cup \partial U_{1}, \\
            I + \mathcal{O}\left( e^{- C n^{2/3}} \right) &
            \textrm{ for } z \in \Sigma_R \setminus (\partial U_{-1} \cup \partial U_0 \cup \partial U_1).
            \end{cases}
\end{align}

From \eqref{appearancez}, \eqref{jumpsr0},\eqref{jumpsr}, and \eqref{IplusZinv} we have when $a = a_n$ as in \eqref{an},
\begin{align} \label{jR2}
    J_{R}(z) & = \left( I - Z_n(z;a_n) +  \frac{Z^{(0)}_n(a_n)}{z} \right)  \nonumber \\
    & \quad \times \left( I + Z_n(z;a_n) + \mathcal{O}\left(n^{-\frac{1}{6}}\right)\right)
    \left( I - \frac{Z^{(0)}_n(a_n)}{z} \right) \nonumber \\
    & = I + \mathcal{O}\left(n^{-\frac{1}{6}}\right), \qquad z \in \partial U_0,
\end{align}
where we also used \eqref{Z0decay} and the fact that $Z_n(z;a_n)$ remains bounded for $|z| = n^{-1/2}$.

\subsection{Conclusion of the steepest descent analysis}
We have now reached the goal of the steepest descent analysis. In the RH problem for $R$ we have
by \eqref{jR1} and \eqref{jR2} that all jump matrices $J_R$ tend to the identity matrix as $n \to \infty$.

Then by standard methods, see e.g.\ \cite{dei99} and also \cite{blku07} for a situation with varying
contours, we have that
$R(z)$ also tends to the identity matrix as $n \to \infty$, at the following rate
\begin{equation}
    R(z) = I + \mathcal{O}\left( \frac{n^{-\frac{1}{6}}}{1 + |z|}\right)
    \textrm{ uniformly for } z \in \mathbb{C} \setminus \Sigma_{R}. \label{estimater} \end{equation}
This concludes the steepest descent analysis of the Riemann-Hilbert problem for $Y$.

\section{Proof of Theorem \ref{maintheorem}}
\label{sec:proof}

Recall that $P_{n,n}(z;a)$ is the $(1,1)$ entry of $Y$, see \eqref{pn1n2fromy},
which we write as
\begin{equation} \label{PnninY}
    P_{n,n}(z;a) = \begin{pmatrix} 1 & 0 & 0 \end{pmatrix} Y(z) \begin{pmatrix}  1 \\ 0 \\ 0 \end{pmatrix}.
\end{equation}
The asymptotic formula \eqref{MehlerHeine} for $P_{n,n}$ will be derived from this by following
the series of transformations $Y \mapsto T \mapsto S \mapsto R$ for $z \in U_{0}$.

In the calculations that follow we assume that $z$ is in the upper part of the lens around $[0,1]$.
The proof for other $z$ in other regions is similar. We obtain  from \eqref{YtoT}, \eqref{TtoS2}, and \eqref{PnninY}
\begin{equation} \label{PnninS}
    P_{n,n}(z;a) = e^{2n(g_1(z) + g_2(z))} \begin{pmatrix} 1 & 0 & 0 \end{pmatrix}
        S(z) \begin{pmatrix}  1 \\ 0 \\ w_2(z)^{-1} e^{2n \varphi_2(z)} \end{pmatrix}.
\end{equation}

For $z \in U_0$ we have by \eqref{definitionp} and \eqref{defR0}
\[ S(z) = R_0(z) P_0(z) =
      R_0(z) E_n(z;a) \Psi\left(n^{\frac{3}{2}} f(z;a); n^{\frac{1}{2}} \tau(z;a) \right)
        e^{n \Lambda(z)} W(z)^{-1}. \]
Inserting this into \eqref{PnninS} and using the expressions \eqref{definitionlambda} and \eqref{definitionw}
for $\Lambda$ and $W$ we obtain
\begin{multline} \label{PnninPsi}
    P_{n,n}(z;a) = z^{-\beta} e^{2n(g_1(z) + g_2(z)} e^{ \frac{2}{3} n ( \varphi_1(z) + \varphi_2(z))}  \begin{pmatrix} 1 & 0 & 0 \end{pmatrix}
         R_0(z) E_n(z;a) \\
         \times \Psi\left(n^{\frac{3}{2}} f(z;a); n^{\frac{1}{2}} \tau(z;a) \right)
         \begin{pmatrix}  1  \\ 0 \\ 1 \end{pmatrix}
\end{multline}
where we used $w_2(z) = z^{\beta} (1-z)^{\gamma} h_2(z)$, see \eqref{w1w2continuation}.
The scalar prefactor in \eqref{PnninPsi} simplifies because of \eqref{definitionvarphi}
and the result is
\begin{multline} \label{PnninPsi2}
    P_{n,n}(z;a) = (-1)^n e^{-\frac{2n}{3} (l_1 + l_2)} z^{-\beta} \begin{pmatrix} 1 & 0 & 0 \end{pmatrix}
         R_0(z) E_n(z;a) \\
         \times \Psi\left(n^{\frac{3}{2}} f(z;a); n^{\frac{1}{2}} \tau(z;a) \right)
         \begin{pmatrix}  1  \\ 0 \\ 1 \end{pmatrix}
\end{multline}

Note that by \eqref{ddinfty}, \eqref{En}, \eqref{definitionz}, and \eqref{defRz}
\begin{multline} \label{R0Enproduct}
    R_0(z) E_n(z;a) = R(z) \left(I + Z_n(z;a) - \frac{Z_n^{(0)}(a)}{z} \right) D_{\infty}^{-1} \widetilde{N}(z;a) \\
    \times  D(z) W(z) A \left(n^{\tfrac{3}{2}} f(z;a); n^{\tfrac{1}{2}} \tau(z;a) \right)^{-1}
\end{multline}
Assuming that $a = -1 + \mathcal O(n^{-1/2})$ and $z = \mathcal O(n^{-3/2})$ as
$n \to \infty$, we have by \eqref{leadingtermsn1} and \eqref{leadingtermsn2}
\begin{equation} \label{Nleadingorder}
    \widetilde{N}(z) = \frac{1}{3 \cdot 2^{1/6}} z^{-\frac{1}{3}} \begin{pmatrix} \sqrt{2} i \\ 1 \\ -1 \end{pmatrix}
    \begin{pmatrix} - \omega & 1 & \omega^2 \end{pmatrix} (I + \mathcal{O}(n^{-1/2})).
    \end{equation}
Using this in the product \eqref{R0Enproduct} we see that the terms $Z_n(z;a) - \frac{Z_n^{(0)}(a)}{z}$
are canceled by the leading order term in \eqref{Nleadingorder}
because of the special form \eqref{definitionz} of $Z_n(z;a)$.

Since $Z_n(z;a) - \frac{Z_n^{(0)}(a)}{z}$ is uniformly bounded, we find that the first four factors
on the right-hand side of \eqref{R0Enproduct} combine to
(we also use \eqref{estimater}),
\begin{equation} \label{firstfactors} \left(I + \mathcal O(n^{-1/6}) \right) \frac{1}{3 \cdot 2^{1/6}} z^{-\frac{1}{3}}
    D_{\infty}^{-1}  \begin{pmatrix} \sqrt{2} i \\ 1 \\ -1 \end{pmatrix}
    \begin{pmatrix} - \omega & 1 & \omega^2 \end{pmatrix}.
    \end{equation}

The behavior of $D(z)$  is determined by \eqref{Dnear0aisminus1}. Then by \eqref{definitionw}
\[ D(z) W(z) =
      C e^{c_1 + c_2} \frac{2^{\tfrac{2}{3} \beta}}{3^{\tfrac{1}{2}(\alpha + \gamma) + \beta}}  z^{\beta/3}
        B_+ (I + \mathcal O(n^{-1/2})) \]
where $B_+$ is the diagonal matrix given by \eqref{definitionomega}. This factor also appears
in the definition \eqref{definitiona} of $A(z; \tau)$. Then the product of the last three factors
on the right-hand side of \eqref{R0Enproduct} gives us
\begin{multline} \label{secondfactors}
    C e^{c_1 + c_2} e^{- \tfrac{1}{6} n \tau(z;a)^2}  \sqrt{\frac{3}{2\pi}} \frac{2^{\tfrac{2}{3} \beta}}{3^{\tfrac{1}{2}(\alpha + \gamma) + \beta}}
    \left( \frac{n^{\frac{3}{2}} f(z;a)}{z} \right)^{- \frac{\beta}{3}} \\
    \times \Omega_+^{-1}
            \begin{pmatrix} \left(n^{\frac{3}{2}} f(z;a)\right)^{-1/3} & 0 & 0 \\ 0 & 1 & 0 \\
                0 & 0 & \left(n^{\frac{3}{2}} f(z;a)\right)^{1/3} \end{pmatrix} (I + \mathcal O(n^{-1/2}))
\end{multline}

Since $\begin{pmatrix} - \omega & 1 & \omega^2 \end{pmatrix} \Omega_+^{-1} = \begin{pmatrix} 0 & 0 & 1 \end{pmatrix}$,
we obtain by multiplying \eqref{firstfactors} and \eqref{secondfactors},
\begin{multline} \label{R0Enestimate}
    R_0(z) E_n(z;a) =
    \frac{  C e^{c_1 + c_2} e^{- \tfrac{1}{6} n \tau(z;a)^2} }{\sqrt{6\pi} \cdot 2^{1/6}}
  \frac{2^{\tfrac{2}{3} \beta}}{3^{\tfrac{1}{2}(\alpha + \gamma) + \beta}}
    \left( \frac{n^{\frac{3}{2}} f(z;a)}{z} \right)^{\frac{1}{3} - \frac{\beta}{3}} \\
    \times
    D_{\infty}^{-1}  \begin{pmatrix} \sqrt{2} i \\ 1 \\ -1 \end{pmatrix}
    \begin{pmatrix} 0 & 0 & 1 \end{pmatrix}
             (I + \mathcal O(n^{-1/6}))
                \end{multline}

Now let us now replace $z$ and $a$ by the scaled variables
\begin{equation} \label{scaling1}
    z_n = \frac{z}{\sqrt{2} n^{3/2}}, \qquad a_n = -1 + \frac{\sqrt{2} \tau}{n^{1/2}}.
    \end{equation}
It then follows from \eqref{scaling1}, and \eqref{lambda20}--\eqref{definitionftau} that
\begin{equation} \label{ftaulimits}
    n^{\frac{3}{2}} f(z_{n};a_{n}) =  z + \mathcal{O}\left( n^{-\frac{1}{2}}\right), \qquad
    n^{\frac{1}{2}} \tau(z_{n};a_{n}) =\tau + \mathcal{O}\left( n^{-\frac{1}{2}}\right),
    \end{equation}
    as $n \to \infty$, where the $\mathcal{O}$-terms hold uniformly for $z$ in a bounded set.
Then we obtain from \eqref{R0Enestimate}
\begin{multline} \label{R0Enscaled} R_0(z_n) E_n(z_n;a_n) =
    \frac{  C e^{c_1 + c_2} e^{- \tfrac{1}{6} \tau^2}}{\sqrt{6\pi}}
  \frac{2^{\tfrac{1}{2} \beta}}{3^{\tfrac{1}{2}(\alpha + \gamma) + \beta}} n^{\frac{1}{2} - \frac{\beta}{2}} \\
    \times
    D_{\infty}^{-1}  \begin{pmatrix} \sqrt{2} i \\ 1 \\ -1 \end{pmatrix}
    \begin{pmatrix} 0 & 0 & 1 \end{pmatrix}
             (I + \mathcal O(n^{-1/6}))
                \end{multline}

We insert \eqref{scaling1}, \eqref{ftaulimits}, and \eqref{R0Enscaled} into
\eqref{PnninPsi2}. We also note that  $D_{\infty}$ depends in an
analytic way on $a$, so that by \eqref{expressiond}
\[ \left(D_{\infty}^{-1}\right)_{1,1} = D_0(\infty;a_n)^{-1} = C^{-1} ( 1 + \mathcal{O}(n^{-1/2})). \]
Then we obtain
\begin{multline}
    P_{n,n}(z_n; a_n) =
    (-1)^n e^{-\frac{2n}{3} (l_1 + l_2)}
    \frac{e^{c_1 + c_2} e^{- \tfrac{1}{6} \tau^2}}{\sqrt{3\pi}}
  \frac{2^{\beta}}{3^{\tfrac{1}{2}(\alpha + \gamma) + \beta}} n^{\frac{1}{2} + \beta}
   \\
    \times i z^{-\beta}
    \begin{pmatrix} 0 & 0 & 1 \end{pmatrix} \Psi\left(z; \tau \right)
         \begin{pmatrix}  1  \\ 0 \\ 1 \end{pmatrix}
             (I + \mathcal O(n^{-1/6}))
             \end{multline}

The constants $l_1$ and $l_2$ depend on $a$ and therefore by \eqref{scaling1} on $\tau$.
One can show that
\begin{equation} \label{l1+l2scaling}
    e^{-\frac{2n}{3}(l_{1} + l_{2})} = \left( \frac{4}{27}\right)^{n} e^{-\sqrt{2}\tau n^{\frac{1}{2}}} e^{-\frac{5\tau^{2}}{6}} \cdot \left( 1 + \mathcal{O}\left(n^{-\frac{1}{2}}\right) \right)
\end{equation}
Thus we proved \eqref{MehlerHeine}
where $C_n$ is given by \eqref{constantCn} and $Q$ is
\begin{equation} \label{QinPsi}
    Q(z) := iz^{-\beta}
    \begin{pmatrix} 0 & 0 & 1 \end{pmatrix} \Psi\left(z; \tau \right)
         \begin{pmatrix}  1  \\ 0 \\ 1 \end{pmatrix}. \end{equation}

Looking at the definition \eqref{definitionpsi1} of $\Psi(z; \tau)$ with $0 < \arg z < \frac{\pi}{4}$,
we find from \eqref{QinPsi}
\begin{equation} \label{Qinq1q2}
    Q(z) = i z^{-\beta}
          \left( e^{2\beta \pi i} q_{1}''(z) + q_{2}''(z) \right)
\end{equation}
Then by the integral representation \eqref{integralrepresentationsq} of $q_1$ and $q_2$
we find by easy contour deformation from \eqref{Qinq1q2}, that for $\re z > 0$,
\begin{equation} \label{QonGamma0}
    Q(z) = i z^{-\beta} \int_{\Gamma_0} t^{-\beta-1} \exp \left( \frac{\tau}{t} - \frac{1}{2t^2} + zt \right) dt,
\end{equation}
with $\Gamma_0$ as in Figure \ref{figuregamma0}.
Making the change of variable $zt \mapsto t$, we arrive at the integral in \eqref{contourQ}, which
shows that $Q$ from \eqref{QinPsi} indeed agrees with the $Q$ defined by \eqref{contourQ} in the theorem.

This completes the proof of Theorem \ref{maintheorem}.

\begin{remark}
The property \eqref{eq:Znsquare} implies that $(I + Z_n(z;a))^{-1} = I  - Z_n(z;a)$ and so by \eqref{appearancez}
\[ (I - Z_n(z;a)) P_0(z) N(z)^{-1} = I + \mathcal{O}(n^{-1/6}) \qquad \textrm{for } z \in \partial U_0. \]
This suggests an alternative approach in which we redefine $P_0(z)$ with the extra factor $I - Z_n(z;a)$
on the left. Using this redefined $P_0$ in the transformation \eqref{defR0} from $S \mapsto R_0$, 
we would arrive at a RH problem for $R_0$ that has all jumps close to the identity matrix, 
but which has a simple pole at $z=0$. This can be included in the RH problem by providing 
a residue condition at $z=0$. The simple form of the residue matrix $Z_n^{(0)}(a)$ allows one to 
remove the pole again in a transformation $R_0 \mapsto R$ as in \cite{dezh94}. The resulting
$R$ is then exactly the same as before. 
 
We thank Percy Deift for this remark.
\end{remark}

\subsection*{Acknowledgement}

Klaas Deschout is supported  by K.U. Leuven
research grant OT/08/33, and the Belgian Interuniversity Attraction
Pole P06/02.

Arno Kuijlaars is supported in part by FWO-Flanders projects G.0427.09
and G.0641.11, by K.U. Leuven
research grant OT/08/33, by the Belgian Interuniversity Attraction
Pole P06/02, and by a grant from the Ministry of Education and Science of
Spain, project code MTM2005-08648-C02-01.

\end{document}